\documentclass[12pt]{amsart}
\usepackage{amssymb}
\usepackage{amsmath}
\usepackage{amsthm}
\usepackage{times}
\usepackage{changepage}
\usepackage{caption}
\usepackage{geometry}
\usepackage[curve,all]{xy}
 \xyoption{curve}
 \xyoption{line}
\xyoption{arc}
\xyoption{color}
\usepackage{array}
\usepackage{enumitem}
\usepackage[dvipdfmx]{graphicx}
\usepackage[dvipdfmx]{graphicx}
\usepackage[dvipdfmx]{color}
\usepackage{color}
\usepackage{amsthm}
\newtheorem{theorem}{Theorem}[section]
\newtheorem{lemma}[theorem]{Lemma}
\newtheorem{corollary}[theorem]{Corollary}
\newtheorem{proposition}[theorem]{Proposition}

\usepackage{subfig}
 
\theoremstyle{definition}

\newtheorem{example}[theorem]{Example}

\theoremstyle{remark}
\newtheorem{remark}[theorem]{Remark}

\numberwithin{equation}{section}

\newcommand{\la}{\langle}
\newcommand{\ra}{\rangle}

\begin{document}
\title [Coble surfaces]{Coble surfaces with finite automorphism group}

\author{Shigeyuki Kond\=o}
\address{Graduate School of Mathematics, Nagoya University, Nagoya,
464-8602, Japan}
\email{kondo@math.nagoya-u.ac.jp}
\thanks{Research of the author is partially supported by
Grant-in-Aid for Scientific Research (A) No.20H00112.}

\begin{abstract}
We classify Coble surfaces with finite automorphism group in characteristic $p\geq 0, p\ne 2$.  There are exactly 9 isomorphism classes of such surfaces.
\end{abstract}

\maketitle

\rightline{In memory of Professor Ernest Vinberg}
\bigskip

\section{Introduction}
\label{sec1}
In this paper we work over an algebraically closed field $k$ of characteristic $p\geq 0, p \ne 2$ and give a classification of Coble surfaces with finite automorphism group.
The notion of a Coble surface comes from the work of Coble \cite{Coble} on birational automorphisms of the plane preserving a rational sextic curve with ten nodes.  Later
Miyanishi called such a surface a Coble surface in his unpublished notes.  
Dolgachev and Zhang \cite{DZ} introduced a more general definition of a Coble surface. In this paper we consider a Coble surface in the strong sense, 
that is, a {\it Coble surface} $S$ is a rational surface with $|-K_S|=\emptyset$ and $|-2K_S|=\{ B_1+\cdots + B_n\}$ where $B_i$ is a non-singular rational curve and $B_i\cap B_j =\emptyset$ $(i\not=j)$. We call $B_1,\ldots, B_n$ the boundary components of $S$.  A Coble surface appears in the classification of involutions of algebraic $K3$ surfaces due to Nikulin \cite{N1} because, 
by definition, there exists a double covering $\pi : X\to S$ branched along $B_1+\cdots + B_n$ such that $X$ is a $K3$ surface.  
A Coble surface also appears as a degeneration of Enriques surfaces (Morrison \cite{Morrison}).  Thus Coble surfaces inherit many properties of Enriques surfaces.
A general Coble surface has an infinite group of automorphisms (e.g. Dolgachev, Kond\=o \cite[Theorem 9.6.1]{DK}) and hence it is natural to classify Coble surfaces with finite automorphism group as in the case of Enriques
surfaces.

First we recall the case of Enriques surfaces. 
The classification of complex Enriques surfaces with finite automorphism group was done by Nikulin \cite{N2} and the author \cite{Ko}.  On the other hand, 
the story of Enriques surfaces $Y$ in characteristic 2 is completely different from the case of 
$p\ne 2$ because there are three types of Enriques surfaces according to 
${\rm Pic}^\tau_Y\cong {\bf Z}/2{\bf Z}, \mu_2$ or $\alpha_2$, where ${\rm Pic}^\tau_Y$ is the group
of divisor classes numerically equivalent to zero (Bombieri and Mumford \cite{BM}).  
They are called classical, $\mu_2$- or $\alpha_2$-surface respectively.
Recently Martin \cite{Martin} gave a classification of Enriques surfaces with finite automorphism group in characteristic $p > 2$ and $p=2$, for $\mu_2$-surfaces, and Katsura, the author and Martin \cite{KKM} have finished the classification for classical and 
$\alpha_2$-surfaces in characteristic 2.  
In the paper \cite{N2}, Nikulin also introduced the notion of $R$-invariant of an Enriques surface $Y$
which is a pair of a root lattice and a 2-elementary finite abelian group.  The $R$-invariant is important because it describes $(-2)$-curves on $Y$ and determines ${\rm Aut}(Y)$ up to finite groups
(we call a non-singular rational curve on a surface with self-intersection number $-m$ 
a $(-m)$-curve).   
For an Enriques surface $Y$, $(-2)$-curves on $Y$ define the nef cone which is a fundamental domain of the subgroup $W(Y)$ of the orthogonal group ${\rm O}({\rm Num}(Y))$ generated by reflections associated with $(-2)$-curves, where ${\rm Num}(Y)$ is the quotient group of the N\'eron-Severi group by the torsion subgroup.  If $W(Y)$ is of finite index in ${\rm O}({\rm Num}(Y))$, then ${\rm Aut}(Y)$ is finite.
Moreover if the automorphism group ${\rm Aut}(Y)$ is finite, then the nef cone is a finite polytope, or equivalently, the number of $(-2)$-curves on $Y$ is finite.  The following Table \ref{Enriques} gives a classification of Enriques surfaces with finite automorphism group in characteristic $p\ne 2$ (Martin \cite{Martin}).  Here 
type means the type of the finite polytope,
$k$ is the number of $(-2)$-curves on $Y$, ${\rm D}_8$ is the dihedral
group of order 8 and $S_n$ is the symmetric group of degree $n$.

\begin{table}[!htb]
\centering
\begin{tabular}{llllllll}
\hline
{\rm Type} & $p$ &  $k$  &{\rm Aut}(Y) & {\rm $R$-invariant} &{\rm Moduli}\\
\hline \hline
{\rm I} & any & 12 & ${\rm D}_8 $ & $(E_8\oplus A_1, \{0\})$ & ${\bf A}^1\setminus \{0,-2^{10}\}$\\ \hline
\rm{II} & any & 12 & $S_4 $  & $(D_9, \{0\})$ &  ${\bf A}^1\setminus \{0,-2^8\}$\\  \hline
\rm{III} & any & 20 & $ ({\bf Z}/4{\bf Z}\times ({\bf Z}/2{\bf Z})^2) \cdot D_8 $  & $(D_8\oplus A_1^{\oplus 2}, ({\bf Z}/2{\bf Z})^2)$ & unique \\  \hline
\rm{IV} & any  & 20 &  $({\bf Z}/2{\bf Z})^4 \cdot ({\bf Z}/5{\bf Z} \cdot {\bf Z}/4{\bf Z})$   & $(D_5^{\oplus 2}, {\bf Z}/2{\bf Z})$ & unique \\  \hline
\rm{V} & $\neq 3$ & 20  & $S_4 \times {\bf Z}/2{\bf Z}$ &    $(E_7\oplus A_2\oplus A_1, {\bf Z}/2{\bf Z})$ & unique \\  \hline
\rm{VI} & $\neq 3,5$ & 20  & $S_5$ & $(E_6\oplus A_4, \{0\})$ & unique\\  \hline
\rm{VII} & $\neq 5$ & 20  & $S_5$ & $(A_9\oplus A_1, {\bf Z}/2{\bf Z})$ & unique \\  \hline
\end{tabular}
\caption{Enriques surfaces with finite automorphism group ($p\ne 2$)}
\label{Enriques}
\end{table}
\noindent
Since the moduli space of Enriques surfaces has dimension 10, Enriques surfaces with finite automorphism group are very rare.

We consider the same problem for Coble surfaces.  
Since the situation of Coble surfaces in characteristic 2 is different from the case of
$p\ne 2$, for example, the canonical covering of a Coble surface in characteristic 2 is inseparable, we focus on the case of $p\ne 2$ in this paper.  
For a Coble surface $S$, following Mukai, instead of its N\'eron-Severi group, we consider a lattice ${\rm CM}(S)$ of rank 10, called Coble-Mukai lattice, which is defined by the orthogonal complement of
the boundary components $B_1, \ldots, B_n$ in the quadratic space generated by ${\rm Pic}(S)$ and ${1\over 2} B_1, \ldots, {1\over 2}B_n$.  
An effective class $\alpha$ in ${\rm CM}(S)$ with $\alpha^2 = -2$ is called an effective root, and an effective root $\alpha$ is called irreducible if 
$|\alpha -\beta|=\emptyset$ for any other effective root $\beta$.  
An effective irreducible root is either a $(-2)$-curve or a divisor of the form 
$2E + {1\over 2} B_i + {1\over 2} B_j$ where $E$ is a $(-1)$-curve with $E\cdot B_i = E\cdot B_j=1$.
The Coble-Mukai lattice ${\rm CM}(S)$ and effective irreducible roots work like 
${\rm Num}(Y)$ and $(-2)$-curves on an Enriques surface $Y$.  

Now we can state the main result of the paper.  We denote by $n$ the number of boundary components 
and by $k$ the number of effective irreducible roots on a Coble surface with finite automorphism group.
The following is the main result of this paper which claims the existence of only 9 isomorphism classes of such Coble surfaces in the 9-dimensional moduli space of Coble surfaces.

\begin{theorem}\label{mainth} 
Coble surfaces with finite automorphism group in characteristic $p\ne 2$ are classified as in the following 
Table $\ref{main}$.  Each type is unique. 
{\rm
\begin{table}[!htb]
\centering
\begin{tabular}{llllllll}
\hline
{\rm Type} & $p$ & $n$ & $k$  &{\rm Aut}(S) &  $R$-invariant  \\
\hline \hline
\rm{I} & any & 1 & 12 & ${\rm D}_8$ & $(E_8\oplus A_1, \{0\})$  \\ \hline
\rm{I} & any & 2 & 12 & ${\rm D}_8$ & $(E_8\oplus A_1^{\oplus 2}, {\bf Z}/2{\bf Z})$  \\ \hline
\rm{II} & any & 1 & 12 & $S_4$  & $(D_9, \{0\})$   \\  \hline
\rm{V} & 3 & 2 & 20  & $S_4 \times {\bf Z}/2{\bf Z}$ &    $(E_7\oplus A_2\oplus A_1^{\oplus 2}, ({\bf Z}/2{\bf Z})^2)$  \\  \hline
\rm{VI} & 5 & 1 & 20  & $S_5$ & $(E_6\oplus A_4, \{0\})$  \\  \hline
\rm{VI} & 3 & 5 & 20  & $S_5$ & $(E_6\oplus D_5, {\bf Z}/2{\bf Z})$  \\  \hline
\rm{VII} & 5 & 1 & 20  & $S_5$ & $(A_9\oplus A_1, {\bf Z}/2{\bf Z})$   \\  \hline
\rm{MI} & 3 & 2 & 40  & ${\rm Aut}(S_6)$ & $(A_5^{\oplus 2}\oplus A_1^{\oplus 2}, ({\bf Z}/2{\bf Z})^{3})$   \\  \hline
\rm{MII} & 3 & 8 & 40  & $(S_4\times S_4)\cdot {\bf Z}/2{\bf Z}$ & $(D_8\oplus A_2^{\oplus 2}, ({\bf Z}/2{\bf Z})^{2})$   \\  \hline
\end{tabular}
\caption{Coble surfaces with finite automorphism group ($p\ne 2$)}
\label{main}
\end{table}
}
\end{theorem}

It is remarkable that Coble surfaces of type I, II appear as degenerations of 1-parameter families of
Enriques surfaces of type I, II, respectively, and Coble surfaces of type V, VI, VII appear as
reductions modulo $p$ of Enriques surfaces of the same type.  
In case of type ${\rm V}$, ${\rm VI}$, ${\rm VII}, {\rm MI}, {\rm MII}$, the covering $K3$ surface $X$ of $S$ is a supersingular $K3$ surface, that is, the Picard number of $X$ is 22 and it exists only in positive characteristic.  
The Coble surfaces of type MI, MII are suggested by Mukai and Ohashi \cite{MO2}.  There are several Enriques surfaces with infinite group of
automorphisms whose nef cones contain the same finite polytope as MI, MII (see Kond\=o \cite{Ko2}, Mukai and Ohashi \cite{MO1}).  In these Enriques surfaces, 
only a part of 40 roots are realized by $(-2)$-curves.

To prove the theorem we employ the method by the author \cite{Ko} and Martin \cite{Martin}.
Consider the natural map $\rho: {\rm Aut}(S) \to {\rm O}({\rm CM}(S))$ which has a finite kernel.  Moreover ${\rm Aut}(S)$ is finite 
if $[{\rm O}({\rm CM}(S)):W(S)]<\infty$ (Proposition \ref{FiniteIndex}).
It is necessary that any genus one fibration on 
$S$ is ``extremal''.  More precisely, a genus one  fibration on $S$ is not minimal in general.  By taking the minimal genus one fibration, it is extremal,
that is, the Mordell-Weil group of its Jacobian fibration is torsion (in case of Enriques surfaces this was first observed by Dolgachev \cite{Do}).
We call an effective irreducible root $r$ is called a special bi-section of a genus one fibration 
if $r\cdot F =2$ where $F$ is a general fiber of the fibration, 
and then the fibration is called special.
And if $S$ contains an effective irreducible root, then $S$ has a special genus one fibration (Lemma \ref{specialfibration}).
This is an analogue of Cossec's theorem in the case of Enriques surfaces (Cossec \cite{C}).

Now the proof goes as follows. Let $S$ be a Coble surface (or an Enriques surface) with finite automorphism group.  
Then any genus one fibration is extremal, and hence $S$ contains an irreducible
effective root.  
We take any special genus one fibration on a Coble surface $S$.
Together with a special bi-section, we can find a new special genus one fibration which 
should be extremal, and hence obtain new effective irreducible roots.
We continue this process to attain a dual graph of effective irreducible roots on $S$ 
whose reflection group
is of finite index in ${\rm O}({\rm CM}(S))$ (or ${\rm O}({\rm Num}(S))$ for Enriques surfaces) or the non-existence of such a surface.  
To check the last condition of the finiteness of the index 
we use a fundamental result (Theorem \ref{Vinberg}) due to Vinberg \cite{V}.  Finally we obtain the dual graphs of
type ${\rm I}, {\rm II}, \ldots , {\rm VII}$ in case of Enriques surfaces and those of type 
${\rm I}, {\rm II}, \ldots,  {\rm VII, MI, MII}$ in
case of Coble surfaces (Theorem \ref{dualgraph}).  
We remark that, as pointed out by Martin \cite[Lemma 10.12]{Martin}, 
there are no special elliptic fibrations on an Enriques surface with singular fibers of type 
${\rm I}_6$, $2{\rm I}_3$, $2{\rm I}_2$ and the one with four singular fibers of type 
$2{\rm I}_3, 2{\rm I}_3, {\rm I}_3, {\rm I}_3$, where $2{\rm I}_3$ means that the fiber of type ${\rm I}_3$ is multiple.
On the other hand the author \cite{Ko}
used some property of automorphisms of complex $K3$ surfaces to exclude these fibrations.
Allowing these two fibrations we obtain the dual graphs of type
MI and MII.  
In this process we use Martin's idea to use the theory of Mordell-Weil lattices (e.g. Sch\"utt and Shioda \cite{SS}).
For a special genus one fibration $p:S\to {\bf P}^1$ with a $(-2)$-curve as its bi-section,
one can construct the Jacobian fibration $j(p)$ of $p$ concretely.  
By applying the theory of Mordell-Weil lattices
one can determine the sections of $j(p)$ and their intersections with reducible fibers.
Thus we can find a new irreducible effective root on $S$ which is helpful to determine the dual graph
of effective irreducible roots.

In Section \ref{sec2} we recall genus one fibrations on a rational surface and Vinberg's theory of hyperbolic reflection groups, and in Section \ref{sec3}, we introduce the 
Coble-Mukai lattice and the $R$-invariant for a Coble surface.  We also discuss properties of genus one fibrations on a Coble surface.  
Section \ref{sec4} is devoted to give examples of Coble surfaces with finite automorphism group.  All examples are given in Dolgachev and Kond\=o \cite{DK}.
In Section \ref{sec5} we determine the dual graphs of effective irreducible roots on a Coble surface with finite automorphism group, and in Section \ref{sec6} the characteristic of the ground field and the number of boundary components of them.  Finally in Section \ref{sec7} we show the uniqueness of each Coble surface in Theorem \ref{mainth}.

\bigskip
\noindent
{\bf Acknowledgements.}
The author thanks Igor Dolgachev and Shigeru Mukai for
teaching him a beauty of Coble surfaces and for stimulating discussions and advices, and
the referee for the careful reading of the manuscript, for pointing out errors and
for many useful suggestions.

\section{Preliminaries}\label{sec2}

A lattice is a free abelian group $L$ of finite rank equipped with 
a non-degenerate symmetric integral bilinear form $\langle \cdot , \cdot \rangle : L \times L \to {\bf Z}$. 
For a lattice $L$ and an integer $m$, we denote by $L(m)$ the free ${\bf Z}$-module $L$ 
with the bilinear form obtained from the one of $L$ by multiplication by $m$. 
The signature of a lattice is the signature of the real vector space $L\otimes {\bf R}$ 
equipped with the symmetric bilinear form extended from the one on $L$ by linearity. A lattice is called even if 
$\langle x, x\rangle \in 2{\bf Z}$ 
for all $x\in L$. 
For an even lattice $L$, we denote by $L^*$ the dual of $L$.  An even lattice $L$ is called unimodular if $L^* \cong L$.
We denote by $L\oplus M$ the orthogonal direct sum of lattices $L$ and $M$, 
and by $L^{\oplus m}$ the orthogonal direct sum of $m$-copies of $L$.
We denote by $U$ the even unimodular lattice of signature $(1,1)$, 
and by $A_m, \ D_n$ or $E_k$ the even negative definite lattice defined by
the Cartan matrix of type $A_m, \ D_n$ or $E_k$ $(m\geq 1, n\geq 4, k=6,7,8)$ respectively.  
A root lattice is a negative definite lattice generated by $(-2)$-vectors. 
It is an orthogonal direct sum of $A_m, \ D_n,\ E_k$. 
We denote by $E_{10}$ an even unimodular lattice of signature $(1,9)$.
It is known that for an Enriques surface $Y$ the lattice ${\rm Num}(Y)$ is isomorphic to $E_{10}$
where ${\rm Num}(Y)$ is the quotient of the N\'eron-Severi group of $Y$ by its torsion subgroup.
Let ${\rm O}(L)$ be the orthogonal group of $L$, that is, the group of isomorphisms of $L$ preserving the bilinear form.
The maps
$$q_L: L^*/L \to {\bf Q}/2{\bf Z}, \quad b_L: L^*/L\times L^*/L \to {\bf Q}/{\bf Z}$$
defined by
$q_L(x \ {\rm mod}\ L) = x^2 \ {\rm mod}\ 2{\bf Z}$, 
$b_L(x \ {\rm mod}\ L, y \ {\rm mod}\ L) = \langle x, y\rangle \ {\rm mod}\ {\bf Z}$
are called the discriminant quadratic form, the discriminant bilinear form of $L$,
respectively.  

\begin{lemma}\label{overlattice}{\rm (Nikulin \cite[Proposition 1.4.1]{N})}
Let $L, L'$ be even lattices such that $L\subset L'$ and $L'/L$ is a finite group.
Then $L'/L (\subset L^*/L)$ is isotropic, that is, $q_L|(L'/L)=0$, and
 $q_{L'} = q_L|(L'/L)^{\perp}/(L'/L)$.
\end{lemma}

We use Kodaira's notation for singular fibers of an elliptic fibration: 
$${\rm I}_n \ (n\geq 1),\ {\rm I}_n^* \ (n\geq 0),\ {\rm II},\ {\rm II}^*,\ {\rm III},\ {\rm III}^*,\ {\rm IV},\  {\rm IV}^*.$$
Also, as the notation of singular fibers, affine Dynkin diagrams $\tilde{A}_m, \tilde{D}_n, 
\tilde{E}_k$ will be used.  The type 
${\rm I}_n$ $(n\geq 2)$, ${\rm III}$, ${\rm IV}$,
${\rm I}^*_{n-4}$ $(n\geq 4)$, ${\rm II}^*$, ${\rm III}^*$ or ${\rm IV}^*$ 
corresponds to
$\tilde{A}_{n-1}$, $\tilde{A}_1$, $\tilde{A}_2$,
$\tilde{D}_{n}$, $\tilde{E}_8$, $\tilde{E}_7$ or $\tilde{E}_6$, respectively.

An elliptic fibration with a section is called a Jacobian fibration.  A Jacobian fibration is extremal
if its Mordell-Weil group is finite.
The following Table \ref{extremal} gives a classification of extremal rational elliptic surfaces in characteristic $p\ne 2$ (e.g. Martin \cite[Table 3]{Martin}).

\begin{table}[!htb]
\centering
\begin{tabular}{lllllll}
\hline
{\rm char}$(k)\ne 2, 3, 5$ & {\rm char}$(k)=5$ & {\rm char}$(k)=3$ \\
\hline \hline
$({\rm II}^*, {\rm II})$ & $({\rm II}^*, {\rm II})$ & $({\rm II}^*)$ \\ \hline
$({\rm III}^*, {\rm III})$ & $({\rm III}^*, {\rm III})$ & $({\rm III}^*, {\rm III})$  \\ \hline
$({\rm IV}^*, {\rm IV})$ & $({\rm IV}^*, {\rm IV})$ & -- \\  \hline
$({\rm I}_0^*, {\rm I}_0^*)$ & $({\rm I}_0^*, {\rm I}_0^*)$ & $({\rm I}_0^*, {\rm I}_0^*)$ \\  \hline
$({\rm II}^*, {\rm I}_1, {\rm I}_1)$ & $({\rm II}^*, {\rm I}_1, {\rm I}_1)$ & $({\rm II}^*, {\rm I}_1)$ \\  \hline
$({\rm III}^*, {\rm I}_2, {\rm I}_1)$ & $({\rm III}^*, {\rm I}_2, {\rm I}_1)$ & $({\rm III}^*, {\rm I}_2, {\rm I}_1)$ \\  \hline
$({\rm IV}^*, {\rm I}_3, {\rm I}_1)$ & $({\rm IV}^*, {\rm I}_3, {\rm I}_1)$ & $({\rm IV}^*, {\rm I}_3)$  \\  \hline
$({\rm I}_4^*, {\rm I}_1, {\rm I}_1)$ & $({\rm I}_4^*, {\rm I}_1, {\rm I}_1)$ & $({\rm I}_4^*, {\rm I}_1, {\rm I}_1)$ \\ \hline
$({\rm I}_2^*, {\rm I}_2, {\rm I}_2)$ & $({\rm I}_2^*, {\rm I}_2, {\rm I}_2)$  & $({\rm I}_2^*, {\rm I}_2, {\rm I}_2)$ \\ \hline 
$({\rm I}_1^*, {\rm I}_4, {\rm I}_1)$ & $({\rm I}_1^*, {\rm I}_4, {\rm I}_1)$ & $({\rm I}_1^*, {\rm I}_4, {\rm I}_1)$  \\  \hline
$({\rm I}_9, {\rm I}_1, {\rm I}_1, {\rm I}_1)$  & $({\rm I}_9, {\rm I}_1, {\rm I}_1, {\rm I}_1)$ & $({\rm I}_9, {\rm II})$ \\  \hline
$({\rm I}_8, {\rm I}_2, {\rm I}_1, {\rm I}_1)$  & $({\rm I}_8, {\rm I}_2, {\rm I}_1, {\rm I}_1)$ & $({\rm I}_8, {\rm I}_2, {\rm I}_1, {\rm I}_1)$ \\  \hline
$({\rm I}_6, {\rm I}_3, {\rm I}_2, {\rm I}_1)$  & $({\rm I}_6, {\rm I}_3, {\rm I}_2, {\rm I}_1)$ & $({\rm I}_6, {\rm I}_3, {\rm III})$ \\  \hline
$({\rm I}_5, {\rm I}_5, {\rm I}_1, {\rm I}_1)$  & $({\rm I}_5, {\rm I}_5, {\rm II})$ & $({\rm I}_5, {\rm I}_5, {\rm I}_1, {\rm I}_1)$ \\  \hline
$({\rm I}_4, {\rm I}_4, {\rm I}_2, {\rm I}_2)$  & $({\rm I}_4, {\rm I}_4, {\rm I}_2, {\rm I}_2)$ & $({\rm I}_4, {\rm I}_4, {\rm I}_2, {\rm I}_2)$ \\  \hline
$({\rm I}_3, {\rm I}_3, {\rm I}_3, {\rm I}_3)$  & $({\rm I}_3, {\rm I}_3, {\rm I}_3, {\rm I}_3)$ & -- \\  \hline
\end{tabular}
\caption{Extremal rational elliptic fibrations ($p\ne 2$)}
\label{extremal}
\end{table}

In characteristic 2 or 3, a quasi-elliptic fibration appears, that is, a fibration whose general fiber
is a rational curve with a cusp.   Any quasi-elliptic fibration on a rational surface is extremal and the reducible fibers are given as in the following Proposition.

\begin{proposition}\label{quasi-ell}{\rm (Ito \cite{Ito2})}
The reducible fibers of a rational quasi-elliptic surface in characteristic $3$ are
$({\rm II}^*),\ ({\rm IV}^*, {\rm IV})$ or $({\rm IV}, {\rm IV}, {\rm IV}, {\rm IV}).$
\end{proposition}

We call an elliptic or a quasi-elliptic fibration a genus one fibration. 
If a genus one fibration on a surface has a multiple fiber with multiplicity 2, for example, 
a multiple fiber of type ${\rm III}$, then we call it a fiber of type $2{\rm III}$ or 
a fiber of type $2\tilde{A}_1$.

Finally we recall the theory of reflection groups in hyperbolic spaces, in particular,
Vinberg's result
which guarantees that a group generated by a finite number of reflections is
of finite index in the orthogonal group. 
Let $L$ be an even lattice of signature $(1,n)$.
Let $\Delta$ be a finite set of $(-2)$-vectors in $L$.
Let $\Gamma$ be the graph of $\Delta$, that is,
$\Delta$ is the set of vertices of $\Gamma$ and two vertices $\delta, \delta'\in \Delta$ are joined 
by $m$-tuple lines if $\langle \delta, \delta'\rangle=m$.
We assume that the cone
$$K(\Gamma) = \{ x \in L\otimes {\bf R} \ : \ \langle x, \delta \rangle \geq 0, \ \delta \in \Delta\}$$
is a strictly convex cone. Such a $\Gamma$ is called non-degenerate.
A connected parabolic subdiagram $\Gamma'$ of $\Gamma$ is an extended Dynkin diagram of 
type $\tilde{A}_m$, $\tilde{D}_n$ or $\tilde{E}_k$ (see Vinberg \cite[p. 345, Table 2]{V}).  
If the number of vertices of $\Gamma'$ is $r+1$, then $r$ is called the rank of $\Gamma'$.  
A disjoint union of connected parabolic subdiagrams is called a parabolic subdiagram of $\Gamma$.  
We denote by $\tilde{K_1}\oplus \tilde{K_2}$ a parabolic subdiagram which is a disjoint union of two
connected parabolic subdiagrams of type $\tilde{K_1}$ and $\tilde{K_2}$.
The rank of a parabolic subdiagram is the sum of 
the ranks of its connected components.  Note that the dual graph of reducible fibers 
of a genus one fibration gives a parabolic subdiagram.  
For example, a singular fiber of type ${\rm III}$, ${\rm IV}$ or ${\rm I}_{n+1}$ 
defines a parabolic subdiagram of type $\tilde{A}_1$, $\tilde{A}_2$ or 
$\tilde{A}_n$ respectively.  
We denote by $W(\Gamma)$ the subgroup of ${\rm O}(L)$ 
generated by reflections $s_\delta: x \to x +\langle x,\delta\rangle \delta$ 
associated with $\delta \in \Gamma$.

\begin{theorem}\label{Vinberg}{\rm (Vinberg \cite[Theorem 2.3]{V})}
Let $\Delta$ be a set of $(-2)$-vectors in a lattice $L$ of signature $(1,n)$ 
and let $\Gamma$ be the graph of $\Delta$.
Assume that $\Delta$ is a finite set, $\Gamma$ is non-degenerate and $\Gamma$ contains 
no $m$-tuple lines with $m \geq 3$.  Then $W(\Gamma)$ is of finite index in ${\rm O}(L)$ 
if and only if every connected parabolic subdiagram of $\Gamma$ is a connected component of some
parabolic subdiagram in $\Gamma$ of rank $n-1$ {\rm (}= the maximal one{\rm )}.
\end{theorem}

\section{Coble surfaces and Coble-Mukai lattices}\label{sec3}

Let $S$ be a Coble surface and $B = B_1+\cdots+B_n$ its anti-bicanonical curve with boundary components $B_1,\ldots, B_n$. 
It is known that $n$ is at most $10$ (e.g. Dolgachev, Kond\=o \cite[Corollary 9.1.5]{DK}).
We denote by $\beta_i$ the divisor class of $B_i$. We have $\beta_i^2 = -4, K_S\cdot \beta_i = 2$. Let $\widetilde{{\rm Pic}}(S)$ be the ${\bf Z}$-submodule of the quadratic vector ${\bf Q}$-space ${\rm Pic}(S)_{\bf Q}$ generated by ${\rm Pic}(S)$ and 
${1 \over 2} \beta_1,\ldots, {1\over 2} \beta_n$. 
Let
$${\rm CM}(S) = \{x\in \widetilde{{\rm Pic}}(S)\ : \ x\cdot \beta_i = 0,\ i = 1,\ldots,n\}.$$
Then ${\rm CM}(S)$ is a lattice (Dolgachev and Kond\=o \cite[\S 9.2]{DK}).
We call ${\rm CM}(S)$ the Coble-Mukai lattice of $S$.
The following is known.

\begin{proposition}\label{CME10}{\rm (Dolgachev and Kond\=o \cite{DK})} 
Let $S$ be a Coble surface with $n$ boundary components.

{\rm (1)}  Assume that $n=1$ or $2$.  
Then ${\rm CM}(S)$ is isomorphic to $E_{10}$.

{\rm (2)} Assume that $k = {\bf C}$.  Then ${\rm CM}(S)$ is isomorphic to $E_{10}$.
\end{proposition}
\begin{proof}
If $n = 1$, then the lattice coincides with $K_{S}^\perp$ and hence it is isomorphic to $E_{10}$.
For $n=2$, see Dolgachev and Kond\=o \cite[Example 9.2.5]{DK}, and for $k={\bf C}$, Dolgachev and Kond\=o \cite[Theorem 9.2.15]{DK}.
\end{proof}

We call an effective class $\alpha\in {\rm CM}(S)$ with $\alpha^2=-2$ an effective root. 
We say that an effective root $\alpha$ is irreducible if $|\alpha-\beta| = \emptyset$ for any other effective root $\beta$.
Since $|-2K_S| = \{B_1+\cdots+B_n\}$, the only curves with negative self-intersection on $S$ are the curves $B_1,\ldots,B_n$, $(-2)$-curves and $(-1)$-curves.
 
\begin{lemma}\label{roots}{\rm (Dolgachev and Kond\=o \cite[Lemma 9.2.1]{DK})} Let $\alpha$ be an effective irreducible root. Then $\alpha$ is either the divisor class of a $(-2)$-curve or the ${\bf Q}$-divisor class of an effective root of the form 
$2e+{1\over 2}\beta_j+{1\over 2} \beta_k$, where $e$ is the class of a $(-1)$-curve $E$ that intersects two different boundary components $B_j,B_k$. 
\end{lemma}

We call an effective irreducible root $\alpha$ a $(-2)$-root (or just $(-2)$-curve) 
if it is represented by
a $(-2)$-curve and a $(-1)$-root if it is represented by 
$2e+{1\over 2}\beta_j+{1\over 2} \beta_k$.

\begin{lemma}\label{roots2}
\begin{itemize}
\item[{\rm (1)}]
Two $(-1)$-roots have intersection number $1$ if and only if the associated $(-1)$-curves do not meet and the roots share precisely one boundary component.
\item[{\rm (2)}]
The intersection number of a $(-1)$-root and a $(-2)$-root is always even.
\end{itemize}
\end{lemma}
\begin{proof}
Let $\alpha = 2E + {1\over 2}B_1 + {1\over 2}B_2$, $\alpha' = 2E' + {1\over 2}B'_1 + {1\over 2}B'_2$ be $(-1)$-roots. Then 
$$\alpha\cdot \alpha'= 4E\cdot E' + 4 - |\{ B_1, B_2, B'_1, B'_2\}|.$$
Thus we have the assertion (1).  The second assertion follows from tha fact that
any $(-2)$-curve does not meet the boundary components. 
\end{proof}

Let ${\rm CM}(S)^+$ be the connected component of $\{ x\in {\rm CM}(S)\otimes {\bf R} \ : \ x^2 > 0\}$
containing a nef class.
Denote by $\overline{{\rm CM}(S)}^+$
the closure of ${\rm CM}(S)^+$ in ${\rm CM}(S)\otimes {\bf R}$.
Let $W(S)$ be the subgroup of ${\rm O}({\rm CM}(S))$ generated by all reflections $s_{\alpha}$, where 
$\alpha$ is an effective irreducible root. 
The group $W(S)$ acts naturally on $\overline{{\rm CM}(S)}^+$.
We denote by $D(S)$ the domain defined by
$$D(S)=\{ x\in \overline{{\rm CM}(S)}^+\ : \ x\cdot \alpha \geq 0 \ {\rm for \ any \ effective \ irreducible \ root} \ \alpha \}.$$
Obviously, the automorphism group ${\rm Aut}(S)$ leaves the set of curves 
$\{B_1,\ldots,B_n\}$ invariant and hence acts on the lattice ${\rm CM}(S)$ and on $D(S)$. 
The following propositions hold.

\begin{proposition}\label{NefCone}{\rm (Dolgachev and Kond\=o \cite[Proposition 9.2.2]{DK})}
Let $x \in \overline{{\rm CM}(S)}^+ \cap {\rm CM}(S)$.  Then $x$ is nef if and only if $x\in D(S)$.
In other word, the intersection of the nef cone of $S$ with $\overline{{\rm CM}(S)}^+$ is a fundamental domain of $W(S)$.
\end{proposition}

\begin{proposition}\label{FiniteIndex}{\rm (Dolgachev and Kond\=o \cite[Theorem 9.8.1]{DK})}
Let $S$ be a Coble surface.  Suppose $W(S)$ is of finite index in ${\rm O}({\rm CM}(S))$.  Then
${\rm Aut}(S)$ is finite.
\end{proposition}

\begin{remark}\label{Vinbergremark}
To prove the finiteness of the index of $W(S)$ in ${\rm O}({\rm CM}(S))$, we use Theorem \ref{Vinberg}.  Note that $\Gamma$ as in Theorem \ref{Vinberg} is automatically non-degenerate if it contains effective irreducible roots appearing in the reducible fibers of an extremal genus one fibration on a Coble surface (or an Enriques surface) and a special bi-section of this fibration. Indeed, these roots generate 
${\rm CM}(S)\otimes {\bf Q}$ (or ${\rm Num}(Y)\otimes {\bf Q}$ for an Enriques surface $Y$) and hence $K(\Gamma)$ is strictly convex.
\end{remark}

By the same proof as in Kond\=o \cite[Proposition 1.10]{Ko}), we have the following.

\begin{proposition}\label{Vinberg2}
Let $S$ be a Coble surface with finite automorphism group and $\Gamma = \{ \alpha_1,\ldots, \alpha_s\}$ a set of mutually different effective irreducible roots on $S$.
Assume that the reflection subgroup $W(\Gamma)$ generated by $s_{\alpha_i}$$(i=1,\ldots, s)$ is of finite index in ${\rm O}({\rm CM}(S))$.  Then $\alpha_i$'s are the only effective irreducible roots on
$S$.
\end{proposition}

Let $S$ be a Coble surface with boundary components $B_1,\ldots, B_n$.
Let $p:S\to {\bf P}^1$ be a genus one fibration on $S$ and let $F$ be a general fiber.
Then the adjunction formula implies that $F\cdot K_S=0$ and hence $F$ is disjoint from $B_1,\ldots, B_n$.
Thus $B_1,\ldots, B_n$ are components of some fibers.
By contracting $(-1)$-curves in fibers, we obtain a relatively minimal rational genus one fibration
$\bar{p}:H\to {\bf P}^1$.  It follows from the proof of Dolgachev and Kond\=o \cite[Proposition 9.1.4]{DK} that $\bar{p}$ is a Halphen surface of index 1 or 2, that is, the fibration $\bar{p}$ is obtained
from a pencil $|3mh - m(p_1+\cdots + p_9)|$ of curves of degree $3m$ on ${\bf P}^2$ 
with points $p_i$ of multiplicity $\geq m$ $(m=1$  or $2)$ where $h$ is a line on ${\bf P}^2$.  
A Halphen surface of index 1 is nothing but a Jacobian genus one fibration.  Moreover we have the following.

\begin{proposition}\label{Halphen}{\rm (Dolgachev and Kond\=o \cite[Proposition 9.1.4]{DK})}
Let $p:S\to {\bf P}^1$ be a genus one fibration on a Coble surface $S$.  
Then $S$ is obtained from a Halphen surface $\bar{p}$ of index $2$ by 
blowing up all singular points 
{\rm (}and their infinitely near points in the case of type ${\rm III, IV}${\rm )}
of one reduced singular fiber or from a Halphen surface of index $1$ by blowing up all singular points {\rm (}and their infinitely near points in the case of type ${\rm III, IV}${\rm )} of two reduced singular fibers $($see Figure {\rm \ref{Halsing}}$)$. 
\end{proposition}

In the above the fiber of a Halphen surface blown up is of type ${\rm II}, {\rm III}, {\rm IV}$ or
${\rm I}_n$ $(n\geq 1)$.  We give more details of the process of blow-up.

\begin{lemma}\label{fibration} We keep the same notation as in the previous Proposition.
Let $F$ be a fiber of $p$ containing a boundary component.  Let  
$\tilde{F}$ be the inverse image of $F$ on the covering $K3$ surface $X$.
If $F$ is obtained by blowing up the singular points of a fiber $\bar{F}$ of $\bar{p}$ 
of type ${\rm II}, {\rm III}, {\rm IV}$ or ${\rm I}_{n}$, then 
$\tilde{F}$ is of type ${\rm IV}, {\rm I}_0^*, {\rm IV}^*$ or ${\rm I}_{2n} (n\geq 1)$,
 respectively.
\end{lemma}
\begin{proof} 
If $\bar{F}$ is a fiber of type ${\rm II}$, then by blowing up the cusp we obtain $F$ 
consisting of a $(-4)$-curve and a $(-1)$-curve meeting at a point with multiplicity 2.  
If $\bar{F}$ is of type 
${\rm III}$, then by blowing up the singular point and its infinitely near point, we obtain $F$ consisting of two $(-4)$-curves $B, B'$, a $(-1)$-curve and a $(-2)$-curve $E$ as in the following Figure \ref{Halsing}, (A):

\begin{figure}[htbp]
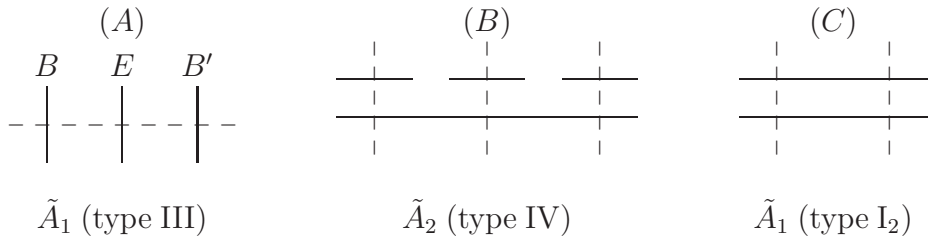

\begin{center}
$
\begin{array}{ccc}
 \hspace{0cm} (A) &  \hspace{1cm} (B) & \hspace{1cm} (C) \\
\xy
(5,0)*{};(35,0)*{}**\dir{--};
(10,-5)*{};(10,5)*{}**\dir{-};
(20,-5)*{};(20,5)*{}**\dir{-};
(30,-5)*{};(30,5)*{}**\dir{-};
(10,8)*{B};(20,8)*{E};(30,8)*{B'};
\endxy 

 & 
 \hspace{1cm}
 
\xy
(5,0)*{};(45,0)*{}**\dir{-};
(10,-5)*{};(10,10)*{}**\dir{--};
(25,-5)*{};(25,10)*{}**\dir{--};
(40,-5)*{};(40,10)*{}**\dir{--};
(20,5)*{};(30,5)*{}**\dir{-};
(5,5)*{};(15,5)*{}**\dir{-};
(35,5)*{};(45,5)*{}**\dir{-};
\endxy 

&
 \hspace{1cm}

\xy
(5,0)*{};(30,0)*{}**\dir{-};
(10,-5)*{};(10,10)*{}**\dir{--};
(25,-5)*{};(25,10)*{}**\dir{--};
(5,5)*{};(30,5)*{}**\dir{-};
\endxy 

  \\
  \hspace{-0cm} {} &  \hspace{0cm} {} &  \hspace{1cm} {} 
 \\
  \hspace{-0cm} \tilde{A}_1 \ ({\rm type \  III}) &  \hspace{1cm} \tilde{A}_2\ ({\rm type \ IV}) & \hspace{1cm} \tilde{A}_1\ ({\rm type \ I_2})


\end{array}
$
\caption{Blowing-up the singular points of a fiber}
\label{Halsing}
\end{center}
\end{figure}
\noindent
In Figure \ref{Halsing}, the dotted lines are $(-1)$-curves.
If $\bar{F}$ is of type 
${\rm IV}$, then by blowing up the singular point and then three intersection points of the exceptional curve and the proper transforms of the components of $\bar{F}$, we obtain $F$ consisting of four $(-4)$-curves 
and three $(-1)$-curves as in Figure \ref{Halsing}, (B).
Finally if $\bar{F}$ is of type ${\rm I}_n$, then by blowing up all singular points, 
we obtain $F$ consisting of $n$ $(-4)$-curves and $n$ $(-1)$-curves as in Figure \ref{Halsing}, (C).  Now the assertion is obvious.

\end{proof}

The following Proposition explains the geometric meaning of Theorem \ref{Vinberg}.

\begin{proposition}\label{Jacobian}{\rm (Dolgachev and Kond\=o \cite[Proposition 9.8.2]{DK})}
Assume that ${\rm Aut}(S)$ is finite.
Then the Jacobian fibration of $\bar{p}$ is extremal.
\end{proposition}

A vector $f\in D(S)\cap {\rm CM}(S)$ is called isotropic if $f^2=0$ and primitive
if $f=mf'$, $m \in {\bf Z}$, $f'\in {\rm CM}(S)$ implies $m=\pm 1$.

\begin{lemma}\label{fibration-isotropic}
The set of genus one fibration on a Coble surface $S$ bijectively corresponds to the set 
of primitive isotropic vectors in $D(S)\cap {\rm CM}(S)$.  
\end{lemma}
\begin{proof}
Obviously the class of a fiber of a genus one fibration gives an  
isotropic vector in $D(S)\cap {\rm CM}(S)$.  
Conversely let $\pi : X\to S$ be the double covering of $S$ branched 
along $B_1+\cdots + B_n$.
Let $f\in D(S)\cap {\rm CM}(S)$ be a primitive isotropic vector.  Then $\pi^*(f)$ is nef (Proposition \ref{NefCone}) and isotropic, and hence $\pi^*(f)$ defines a genus one fibration $\tilde{p}:X\to {\bf P}^1$ on the $K3$ surface $X$.  The fibration $\tilde{p}$ induces a genus one fibration $p:S\to {\bf P}^1$.
\end{proof}

Let $\pi : X\to S$ be the double covering of $S$ branched 
along $B_1+\cdots + B_n$ with covering involution $\sigma$.
Let $f : S \to {\bf P}^1$ be a genus one fibration and let $\tilde{f}: X \to {\bf P}^1$ 
be the pull back of $f$.  Then $\sigma$ acts non-trivially on the base of $\tilde{f}$ by the same proof of \cite[Proposition 9.34]{Ko3}.
It now follows from Proposition \ref{Halphen} that 
there are exactly two $\sigma$-invariant fibers $\tilde{F}_1, \tilde{F}_2$ 
of $\tilde{f}$, and the boundary components are contained in $F_1$ or $F_2$
(In case that $f$ is obtained from a Halphen surface of index 2
(resp. index 1), $F_1, F_2$ are the multiple fiber and the fiber containing the boundary components (resp. the two fibers containing the boundary components)).
There are two possibilities: one of $F_1, F_2$ contains all $B_1,\ldots, B_n$ or not.

An effective irreducible root $r$ is called a special bi-section of 
a genus one fibration $p:S\to {\bf P}^1$
if $r\cdot F=2$ where $F$ is a general fiber of $p$.  In this case the fibration $p$ is called special.  In case of Enriques surfaces, the following is known.

\begin{proposition}\label{Cossec}{\rm (Cossec \cite[Theorem 4]{C})}
If an Enriques surface $Y$ contains a $(-2)$-curve, then there exists a special elliptic fibration on $Y$.
\end{proposition}

The following is the analogue of Cossec's theorem for Coble surfaces.

\begin{lemma}\label{specialfibration} Let $S$ be a Coble surface.  
Assume that $S$ has an irreducible effective root.  
Then there exists a special genus one fibration on $S$.  
\end{lemma}
\begin{proof}
Let $n$ be the number of boundary components of $S$.
In case of $n=1,2$, ${\rm CM}(S)$ is isomorphic to $E_{10}$ (Proposition \ref{CME10}).
The proof of Proposition \ref{Cossec} depends on this fact and the analogue of Proposition \ref{NefCone}, and hence works well in case of Coble surfaces, too.
Thus the assertion in case of $n=1,2$ follows.

Now assume $n \geq 3$.
The relatively minimal model of a genus one fibration $p:S\to {\bf P}^1$ is a 
Halphen surface $H$ of index 1 or 2,  
and hence it has a section or a bi-section.  Denote it by $E'$ and let $E$ be the proper
transform of $E'$ on $S$.

Assume that $E^2=-2$.  Then $E$ is obtained by blowing up a point $p$ of $E'$.  This means that $p$ is a singular point of a fiber.  Hence the intersection number of $E$ and the obtained fiber is 2, and hence $E$ is a special bi-section. 

Assume $E^2=-1$.  Then $E\cdot B=E\cdot B'=1$ or $E\cdot B=2$ where $B, B'$ are
boundary components. 

First consider the case $E$ meets $B$ and $B'$.  
If $H$ is a Halphen surface of index 1, then $E'=E$ is a section of $p$ and hence 
the effective irreducible root $\alpha = 2E + {1\over 2}B + {1\over 2}B'$ is 
a special bi-section of $p$.
If $H$ is of index 2, a fiber $F$ contains both $B$ and $B'$ (Proposition \ref{Halphen}). There exists a fiber $F$ that contains boundary components $B_1=B, \ldots , B_k=B'$ and 
$(-1)$-curves $E_1, \ldots, E_{k-1}$ such that
$E_1\cdot B_1=E_1\cdot B_2=1$, \ldots, $E_{k-1}\cdot B_{k-1}=
E_{k-1}\cdot B_k=1$ $(k\geq 2)$.  We assume that $k$ is the smallest possible one.
Since $E+E_1+\cdots + E_{k-1}+{1\over 2}(B_1+\cdots +B_k) \in {\rm CM}(S)$ is primitive, 
nef and isotropic, it
defines a genus one fibration $p':S\to {\bf P}^1$ (Lemma \ref{fibration-isotropic}).
Then, by the assumption $n\geq 3$ and the minimality of $k$, there exist a boundary component $B_{k+1}$ and $(-1)$-curve $E_k$
such that $E_k\cdot B_i=E_k\cdot B_{k+1}=1$ for some $i$ $(1\leq i\leq k)$.  Then $2E_k+{1\over 2}B_i +{1\over 2}B_{k+1}$ is
a special bi-section of $p'$.

Next consider the case $E$ meets $B$ with multiplicity 2.  
In this case, $H$ is of index 2 because $E$ gives a bi-section of $p$.  
Let $F$ be the fiber of $p$ containing $B$.   Then all boundary components are contained in $F$.  By the assumption $n\geq 3$, there exist a boundary component $B'$ and a $(-1)$-curve $E_1$
such that $E_1\cdot B=E_1\cdot B'=1$.  
The divisor $2E+B$ is nef, $(2E+B)^2=0$ and hence $E+{1\over 2}B (\in {\rm CM}(S))$ 
defines a genus one fibration $p':S\to {\bf P}^1$ (Lemma \ref{fibration-isotropic}).
Then $2E_1+{1\over 2}B+{1\over 2}B'$ is a special bi-section of $p'$.
\end{proof}

\begin{remark}\label{notationFib}
As mentioned above, a genus one fibration on a Coble surface is not relatively minimal. 
However by using $(-1)$-roots instead of $(-1)$-curves and $(-4)$-curves, we can use
the symbols $\tilde{A}_m, \tilde{D}_n, \tilde{E}_k$ for singular fibers.  The vertices of these
extended Dynkin diagram are represented by $(-1)$- and $(-2)$-roots.
\end{remark}

\medskip
The following gives a relation between a special genus one fibration on a Coble surface and 
its Jacobian fibration which was proved for Enriques surfaces over the complex numbers 
by the author \cite[Lemma 2.6]{Ko} and
was generalized to positive characteristic by Martin \cite[Lemma 2.17]{Martin}.
We omit the proof which is similar to the one for Enriques surfaces.

\begin{proposition}\label{Jacobian2}
Let $S$ be a Coble surface, $X$ the covering $K3$ surface and $\sigma$ 
the covering involution.  Let $p:S\to {\bf P}^1$ be a special genus one fibration 
with a $(-2)$-curve $N$ as a bi-section. 
Let $\tilde{p}:X\to {\bf P}^1$ be the induced genus one fibration and $J(p):J(S)\to {\bf P}^1$ the Jacobian fibration of $p$. 
Let $N^+, N^-$ be the inverse image of $N$ on $X$ which are disjoint $(-2)$-curves.  Consider $N^+$ as the zero section of
$\tilde{p}$ and denote by $t$ the translation by $N^-$.  Then $j(\sigma)= t\circ \sigma$ is an involution and the quotient $X/\langle j(\sigma)\rangle$ is isomorphic to $J(S)$.
\end{proposition}

\noindent
Note that if a bi-section $N$ is a $(-1)$-root, then $p$ is obtained from
a Jacobian fibration by blowing up.

In case of positive characteristics, the covering $K3$ might be supersingular, that is,
its Picard number is 22.  The following is well-known.

\begin{proposition}\label{SSK3Pic}{\rm (Artin \cite{Artin})} Let $X$ be a supersingular $K3$ surface in characteristic $p>0$.  Then ${\rm Pic}(X)^*/{\rm Pic}(X)\cong ({\bf Z}/p{\bf Z})^{2\sigma}$ 
$(1\leq \sigma\leq 10)$.  The number $\sigma$ is called Artin invariant.
\end{proposition}

Finally we define $R$-invariants for Coble surfaces.  
Let $S$ be a Coble surface with $n$ boundary components $B_1,\ldots, B_n$. 
Let $\pi:X\to S$ be the double covering branched along $B_1, \ldots, B_n$ and $\sigma$ the covering involution.
For any curve $C$ on $S$ we denote by $\tilde{C}$ the inverse image of $C$.
Define $L^+ = \pi^*({\rm CM}(S))$.
Denote by $L^-$ the orthogonal complement of $L^+$ in ${\rm Pic}(X)$.
Note that $(-2)$-curves $\tilde{B}_1,\ldots , \tilde{B}_n$ are contained in $L^-$.
Define
$$h^{\pm} = \left\{ \delta^{\pm}\in L^{\pm} \ : \ \exists\ \delta^{\mp}\in L^{\mp}, \ (\delta^{\pm})^2 = -4, 
\ {\delta^++\delta^-\over 2} \in {\rm Pic}(X) \right\}.$$
Let $\la h^-\ra$ be the sublattice of $L^-$ generated by $h^-$.  
Then $\la h^-\ra=K(2)$ where $K$ is a root lattice.  

For an even lattice $L$, we define a quadratic form $q_{L/2L}$ and a symmetric bilinear form $f_{L/2L}$ on $L/2L$ by
$$q_{L/2L}(x + 2L)= {1\over 2} \langle x, x\rangle \ {\rm mod}\ 2{\bf Z},\quad f_{L/2L}(x + 2L, y+2L)= \langle x, y\rangle \ {\rm mod}\ 2{\bf Z}.$$
The kernel of $f_{L/2L}$ is the subspace of all $x$ such that $f_{L/2L}(x,y)=0$ for all $y\in L/2L$ and the
kernel of $q_{L/2L}$ is the subspace of all $x$ in the kernel of $f_{L/2L}$ for which also $q_{L/2L}(x)=0$.
We define the nullity and the rank of $q_{L/2L}$ to be the dimension and the codimension of 
its kernel, and call $q_{L/2L}$ non-singular when its nullity is zero.  
A subspace is called isotropic if $q_{L/2L}(x)$ vanishes for all $x$ in the subspace (see \cite[page xi]{ATLAS}).
Then we have a homomorphism
$$\gamma : K/2K\cong K(2)/2K(2) \to L^+/2L^+ \cong {\rm CM}(S)/2{\rm CM}(S)$$
defined by
$$\gamma (\delta^-\ {\rm mod}\ 2) = \delta^+\ {\rm mod}\ 2.$$
We define a subgroup $H$ of $K/2K$ by the kernel of $\gamma$ which is isotropic with respect to $q_{K/2K}$.
In this case, we allow that the overlattice $K_H=\{ x\in K\otimes {\bf Q}\ : \ 2x \in H\}$
contains $(-1)$-vectors contrary to the case of Enriques surfaces.
We call the pair $(K, H)$ the $R$-invariant of $S$.
  
Recall that there are two types of effective irreducible roots.  
One is a $(-2)$-root (= a $(-2)$-curve) and the other is a $(-1)$-root $2E + {1\over 2}B_i + {1\over 2}B_j$, where $E$ is a $(-1)$-curve that intersects $B_i$ and $B_j$.  If $C$ is a $(-2)$-curve, then $\pi^*(C)=\tilde{C}^++\tilde{C}^-$ 
where $\tilde{C}^+, \tilde{C}^-$ are disjoint $(-2)$-curves on $X$ with 
$\sigma(\tilde{C}^+)=\tilde{C}^-$. If $E$ is a $(-1)$-curve meeting two boundary components, then $\pi^*(E)=\tilde{E}$ is a $(-2)$-curve.  
If $C=B_i$, then $\pi^*(B_i) =2\tilde{B}_i$.
Therefore $\pi^*(2E + {1\over 2}B_i + {1\over 2}B_j)=2\tilde{E}+\tilde{B}_i+\tilde{B}_j$.
 We associate $\delta^+ = \tilde{C}^++\tilde{C}^-$ and $\delta^-=\tilde{C}^+-\tilde{C}^-$ to the 
 $(-2)$-root $C$, and associate $\delta^+= 2E + \tilde{B}_i + \tilde{B}_j$ and 
 $\delta^-=\tilde{B}_i \pm \tilde{B}_j$
 to the $(-1)$-root $2E + {1\over 2}B_i + {1\over 2}B_j$.  
Note that $K_H$ contains
$(-1)$-vectors corresponding to $\tilde{B}_1,\ldots , \tilde{B}_n$ when $n\geq 2$.
For more details and examples, we refer the reader to Dolgachev and Kond\=o \cite[Chapter 9]{DK}.

\section{Examples}\label{sec4}

In this section we present examples of Coble surfaces with finite automorphism group.
All of them are given in Dolgachev and Kond\=o \cite[\S 9.8]{DK}.  
Here we also give another description
of Coble surfaces of type I, II, V and of VI with $n=5$.  

Let $(u_0:u_1, v_0:v_1)$ be bi-homogeneous coordinates of ${\bf P}^1\times {\bf P}^1$ and let
$\phi$ be the involution of ${\bf P}^1\times {\bf P}^1$ defined by 
$\phi (u_0:u_1, v_0:v_1)= (u_0:-u_1, v_0:-v_1)$.  

\begin{example}\label{TypeI2} \ {\bf A Coble surface of type I with 2 boundary components.}

In Kond\=o \cite[Example I]{Ko}, the author gave a complex Enriques surface $Y$ of type I whose covering $K3$ surface $X$ is the minimal resolution of the double covering of ${\bf P}^1\times {\bf P}^1$ branched along the $\phi$-invariant divisor $C+ \sum_{i=1}^4L_i$ of bi-degree $(4,4)$, where $C$ and $L_i$ are defined by
$$C:(2u_0^2 - u_1^2)(v_0^2 - v_1^2)=(2\alpha v_0^2+(1-2\alpha)v_1^2)(u_0^2-u_1^2), \ \alpha \in k, \alpha \ne 1, 1/2, 3/2 ;$$
$$L_1: u_0/u_1=1; \ L_2: u_0/u_1=-1; \ L_3: v_0/v_1=1;\ L_4: v_0/v_1=-1.$$
Let $\sigma$ be the lift of $\phi$ to $X$ which is a fixed point free involution of $X$ because the fixed points of $\phi$ are not contained in the branch divisor.  Then
$Y\cong X/\la \sigma \ra$.

This example exists in positive characteristic $p\ne 2$.
In case of $\alpha = 1$, the curve $C$ splits into the two quadrics 
$$Q_1: u_1v_0-u_0v_1=0, \quad Q_2: u_1v_0+u_0v_1=0$$
and the involution $\phi$ fixes two intersection points $q_1=(1:0, 1:0), q_2=(0:1,0:1)$ of $Q_1$ and $Q_2$.  
The involution $\sigma$ of $X$ has two fixed curves which are exceptional curves over 
$q_1, q_2$.  Thus we obtain the Coble surface $S= X/\la \sigma\ra$ with two boundary
components $B_1, B_2$.  In case of Enriques surfaces of type I, 
the images of $Q_1$, $Q_2$ are $(-2)$-curves meeting
together at two points transversely.  On the other hand, on the Coble surface $S$, 
the images of two curves $Q_1$, $Q_2$ are $(-1)$-curves $E_1, E_2$, and hence 
two $(-1)$-roots $2E_i+{1\over 2}B_1 + {1\over 2}B_2$ \ $(i=1,2)$ appear. 
The surface $S$ contains twelve effective irreducible roots forming the same dual graph 
$\Gamma_{\rm I}$ of an Enriques surface of type I as in the following Figure \ref{typeI}.  
In Figure \ref{typeI} a circle denotes a $(-1)$-root
and the remaining ten vertices are $(-2)$-roots.

\begin{figure}[htbp]
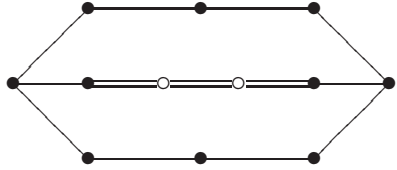

\begin{center}
 \xy
(-30,0)*{};
(20,0)*{\circ},(30,0)*{\circ}
@={(00,0),(10,0),(40,0),(50,0),(25,10),(10,10),(40,10),(10,-10),(25,-10),(40,-10)}@@{*{\bullet}};
(00,0)*{};(10,0)*{}**\dir{-};
(10,0)*{};(19,0)*{}**\dir{=};
(21,0)*{};(29,0)*{}**\dir{=};
(31,0)*{};(40,0)*{}**\dir{=};
(40,0)*{};(50,0)*{}**\dir{-};
(00,0)*{};(10,10)*{}**\dir{-};
(10,10)*{};(25,10)*{}**\dir{-};
(25,10)*{};(40,10)*{}**\dir{-};
(40,10)*{};(50,0)*{}**\dir{-};
(00,0)*{};(10,-10)*{}**\dir{-};
(10,-10)*{};(25,-10)*{}**\dir{-};
(25,-10)*{};(40,-10)*{}**\dir{-};
(40,-10)*{};(50,0)*{}**\dir{-};
\endxy 
 \caption{The dual graph $\Gamma_{\rm I}$ of Type I}
 \label{typeI}
\end{center}
\end{figure}

\noindent
The dual graph satisfies the condition in Theorem \ref{Vinberg} and hence
${\rm Aut}(S)$ is finite (Proposition \ref{FiniteIndex}).  Moreover it follows from Proposition \ref{Vinberg2} that
$S$ has exactly twelve effective irreducible roots.
By considering Jacobian fibrations (see \cite[Theorem 3.1.1]{Ko}), we can prove
${\rm Aut}(S)$ is isomorphic to the dihedral group ${\rm D}_8$ of order 8.

\begin{proposition}\label{typeI2prop}{\rm (Dolgachev and Kond\=o \cite[Example 9.8.4]{DK})}
Let $S$ be the Coble surface of type ${\rm I}$ with two boundary components.
Then ${\rm Aut}(S)\cong {\rm D}_8$ and 
the $R$-invariant of $S$ is $(E_8\oplus A_1^{\oplus 2}, {\bf Z}/2{\bf Z})$.
\end{proposition}

\end{example}

\begin{example}\label{TypeI1} {\bf A Coble surface of type I with one boundary component.}

We continue the same notation as in Example \ref{TypeI2}.
In case of $\alpha = 1/2$ or $3/2$, the curve $C$ has a node at $q=(0:1, 1:0)$ which is a fixed point of $\phi$.  The involution $\sigma$ has a fixed curve which is the exceptional curve over $q$.
Thus we obtain the Coble surface $S= X/\la \sigma\ra$ with a boundary component.  
In this case, the surface $S$ has twelve $(-2)$-curves forming the dual graph $\Gamma_{\rm I}$ given in Figure \ref{typeI} as in the case of Enriques surfaces of type I. 

\begin{proposition}\label{typeI1prop}{\rm (Dolgachev and Kond\=o \cite[Example 9.8.5]{DK})}
Let $S$ be the Coble surface of type ${\rm I}$ with one boundary component.
The surface $S$ has the automorphism group ${\rm Aut}(S)$ isomorphic to the dihedral group ${\rm D}_8$ of order $8$.  The $R$-invariant of $S$ is $(E_8\oplus A_1, \{ 0 \})$.
\end{proposition}

\end{example}

\begin{example}\label{TypeII} {\bf A Coble surface of type II with one boundary component.}

We use the same notation as in Example \ref{TypeI2}.
In Kond\=o \cite[Example II]{Ko}, the author gave a complex Enriques surface $Y$ of type II whose covering $K3$ surface $X$ is the minimal resolution of the double covering of ${\bf P}^1\times {\bf P}^1$ branched along the divisor $C + \sum_{i=1}^2L_i^{\pm}$ of bi-degree $(4,4)$, where $C$ and $L_i^{\pm}$ are defined by
$$C: (v_0^2 - v_1^2)u_0^2 = (v_0^2 + \alpha v_1^2)u_1^2, \ \alpha \in k, \alpha \ne 0, -1;$$
$$L_1^{\pm}: u_0/u_1=\pm 1; \ L_2^{\pm}: v_0/v_1=\pm 1.$$

This example also exists in positive characteristic $p\ne 2$.
In case of $\alpha = 0$, the curve $C$ has a node at $q=(0:1, 0:1)$ which is a fixed point of $\phi$.  The involution $\sigma$ has a fixed curve which is the exceptional curve over $q$.
Thus we obtain a Coble surface $S= X/\la \sigma\ra$ with a boundary component.  
The surface $S$ has twelve $(-2)$-curves forming the dual graph $\Gamma_{\rm II}$ given in the following Figure \ref{typeII}:

\begin{figure}[htbp]
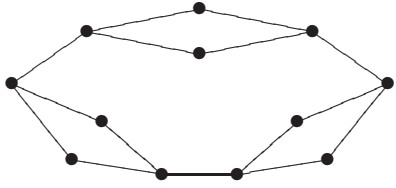

\begin{center}
 \xy
(-30,0)*{};
@={(00,0),(50,0),(20,-12),(30,-12),(8,-10),(42,-10),(12,-5),(38,-5),(10,7),(25,10),(40,7),(25,4)}@@{*{\bullet}};
(42,-10)*{};(50,0)*{}**\dir{-};
(38,-5)*{};(50,0)*{}**\dir{-};
(10,7)*{};(25,10)*{}**\dir{-};
(10,7)*{};(25,4)*{}**\dir{-};
(25,10)*{};(40,7)*{}**\dir{-};
(25,4)*{};(40,7)*{}**\dir{-};
(00,0)*{};(8,-10)*{}**\dir{-};
(00,0)*{};(12,-5)*{}**\dir{-};
(8,-10)*{};(20,-12)*{}**\dir{-};
(12,-5)*{};(20,-12)*{}**\dir{-};
(00,0)*{};(10,7)*{}**\dir{-};
(50,0)*{};(40,7)*{}**\dir{-};
(30,-12)*{};(38,-5)*{}**\dir{-};
(20,-12)*{};(30,-12)*{}**\dir{-};
(30,-12)*{};(42,-10)*{}**\dir{-};
\endxy 
 \caption{The dual graph $\Gamma_{\rm II}$ of Type II}
 \label{typeII}
\end{center}
\end{figure}

\noindent
The dual graph satisfies the condition in Theorem \ref{Vinberg} and hence
${\rm Aut}(S)$ is finite (Proposition \ref{FiniteIndex}).  Moreover it follows from Proposition \ref{Vinberg2} that
$S$ has exactly twelve effective irreducible roots.
By considering Jacobian fibrations (see \cite[Theorem 3.2.1]{Ko}), we can prove
${\rm Aut}(S)$ is isomorphic to the symmetric group $S_4$ of degree 4.

In case of $\alpha = -1$, the divisor $C + \sum_{i=1}^2L_i^{\pm}$ is not reduced.

\begin{proposition}\label{typeIIprop}{\rm (Dolgachev and Kond\=o \cite[Example 9.8.6]{DK})}
Let $S$ be the Coble surface of type ${\rm II}$ with one boundary component.
Then ${\rm Aut}(S)\cong S_4$ and the $R$-invariant of $S$ is $(D_9, \{ 0 \})$.
\end{proposition}

\end{example}

\begin{example}\label{TypeV} {\bf A Coble surface of type V in characteristic 3 with 2 boundary components.}

We use the same notation as in Example \ref{TypeI2}.
In Kond\=o \cite[Example V]{Ko}, the author gave a complex Enriques surface $Y$ of type {\rm V} whose covering $K3$ surface $X$ is the minimal resolution of the double covering of ${\bf P}^1\times {\bf P}^1$ branched along the divisor $C_+ + C_- + \sum_{i=1}^2L_i^{\pm}$ of bi-degree $(4,4)$, where $C_{\pm}, L_i^{\pm}$ are defined by
$$C_{\pm}: (v_0 \pm v_1)(u_0 \pm u_1) = -2(u_0 \mp u_1)v_0;\quad  
L_1^{\pm}: u_0/u_1=\pm 1; \ L_2^{\pm}: v_0/v_1=\pm 1.$$
Also we used four curves
$$F_{1,\pm}: u_1/u_0=\pm \sqrt{-3}; \quad F_{2,\pm}: u_1/u_0=\pm \sqrt{-3}$$
to obtain $(-2)$-curves on the Enriques surface.

Now we assume that the ground field $k$ is an algebraically closed field in characteristic 3.
Then the curves $C_+$ and $C_-$ are given by the equations $u_0v_1 - u_1v_0 + u_1v_1=0$ and $u_0v_1-u_1v_0 - u_1v_1=0$, respectively, meeting at $q=(1:0, 1:0)$ with multiplicity 2.  Note that $q$ is a fixed point of $\phi$.  
The double covering of ${\bf P}^1\times {\bf P}^1$ has a rational double point of type $A_3$ over $q$ 
and hence
$X$ contains three $(-2)$-curves $E_0, E_1, E_2$ with $E_0\cdot E_1=E_0\cdot E_2=1$.
The involution $\sigma$ fixes $E_1, E_2$ point-wisely.
Thus we obtain a Coble surface $S= X/\la \sigma\ra$ with two boundary components $B_1, B_2$.  

Instead of $F_{1,\pm}, F_{2,\pm}$, we consider the curves
$$F_1: u_1=0;\quad F_2: v_1=0;\quad M: u_0v_1 + u_1v_0=0.$$
In fact $M$ did not appear in \cite[Fig. 5.1]{Ko}, but it corresponds to the $(-2)$-curve $E_{20}$ in
\cite[Fig. 5.5]{Ko}.
These three curves pass through the point $q$ and the images on $S$ are $(-1)$-curves meeting $B_1$ 
and $B_2$.  Thus we have four $(-1)$-roots with the dual graph in the following 
Figure \ref{typeV}, (C), which are associated with the images of $E_0$, $F_1, F_2, M$.
Moreover $S$ contains 16 $(-2)$-curves forming the dual graph 
in the Figure \ref{typeV}, (A), (B).  These 16 $(-2)$-curves are the same as in the case of the Enriques surface of type V.  Thus $S$ has the same dual graph $\Gamma_{\rm V}$ of 20 effective irreducible roots.
For the incidence relation between curves in (A), (B) and (C), we refer the reader to
Kond\=o \cite[Fig. 5.5]{Ko}.

\begin{figure}[htbp]
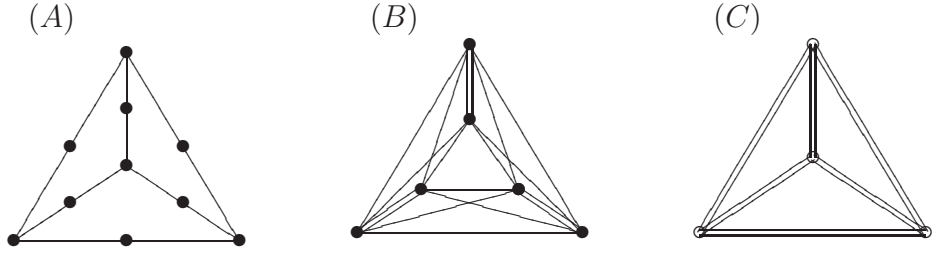

\begin{center}
$
\begin{array}{ccc}
 \hspace{-2cm} (A) &  \hspace{-1cm} (B) &  \hspace{-1cm} (C)\\

\xy
@={(00,0),(15,0),(30,0),(15,25),(15,10),(15,17.5),(7.5,5),(22.5,5),(7.5,12.5),(22.5,12.5)}@@{*{\bullet}};
(00,0)*{};(15,0)*{}**\dir{-};
(15,0)*{};(30,0)*{}**\dir{-};
(00,0)*{};(7.5,5)*{}**\dir{-};
(00,0)*{};(7.5,12.5)*{}**\dir{-};
(7.5,5)*{};(15,10)*{}**\dir{-};
(7.5,12.5)*{};(15,25)*{}**\dir{-};
(15,10)*{};(15,25)*{}**\dir{-};
(30,0)*{};(15,25)*{}**\dir{-};
(30,0)*{};(15,10)*{}**\dir{-};
\endxy 

& 
 \hspace{1cm}

\xy
@={(00,0),(30,0),(15,25),(15,15),(8.5,5.7),(21.5,5.7),}@@{*{\bullet}};
(00,0)*{};(30,0)*{}**\dir{-};
(00,0)*{};(8.5,5.7)*{}**\dir{=};
(00,0)*{};(15,25)*{}**\dir{-};
(00,0)*{};(15,15)*{}**\dir{-};
(00,0)*{};(21.5,5.7)*{}**\dir{-};
(15,25)*{};(15,15)*{}**\dir{=};
(15,25)*{};(8.5,5.7)*{}**\dir{-};
(15,25)*{};(21.5,5.7)*{}**\dir{-};
(30,0)*{};(15,25)*{}**\dir{-};
(30,0)*{};(15,15)*{}**\dir{-};
(30,0)*{};(21.5,5.7)*{}**\dir{=};
(30,0)*{};(8.5,5.7)*{}**\dir{-};
(8.5,5.7)*{};(21.5,5.7)*{}**\dir{-};
(15,15)*{};(21.5,5.7)*{}**\dir{-};
(15,15)*{};(8.5,5.7)*{}**\dir{-};
\endxy

 & 
 \hspace{1cm}
 
\xy
@={(00,0),(30,0),(15,25),(15,10)}@@{*{\circ}};
(00,0)*{};(30,0)*{}**\dir{=};
(00,0)*{};(15,10)*{}**\dir{=};
(00,0)*{};(15,25)*{}**\dir{=};
(30,0)*{};(15,10)*{}**\dir{=};
(30,0)*{};(15,25)*{}**\dir{=};
(15,10)*{};(15,25)*{}**\dir{=};
\endxy 
\end{array}
$
 \caption{The dual graph $\Gamma_{\rm V}$ of Type V}
 \label{typeV}
\end{center}
\end{figure}

\noindent
The dual graph satisfies the condition in Theorem \ref{Vinberg} and hence
${\rm Aut}(S)$ is finite (Proposition \ref{FiniteIndex}).  Moreover it follows from Proposition \ref{Vinberg2} that
$S$ has exactly twenty effective irreducible roots.
By considering Jacobian fibrations, we can prove
${\rm Aut}(S)$ is isomorphic to $S_4\times {\bf Z}/2{\bf Z}$.

\begin{proposition}\label{typeVprop}{\rm (Dolgachev and Kond\=o \cite[Theorem 9.8.8]{DK})}
Let $S$ be the Coble surface of type {\rm V} with two boundary components in characteristic $3$.
Then ${\rm Aut}(S)\cong S_4\times {\bf Z}/2{\bf Z}$ and  
the $R$-invariant of $S$ is $(E_7\oplus A_2\oplus A_1^{\oplus 2}, ({\bf Z}/2{\bf Z})^2)$.
\end{proposition}

\end{example}

\begin{example}\label{TypeVI3} {\bf A Coble surface of type VI in characteristic 3 with 5 boundary components.}

The Enriques surface of type VI in characteristic $p\ne 3, 5$ can be obtained as the quotient of the Hessian quartic surface associated 
with the Clebsch diagonal cubic surface $S$ in ${\bf P}^4$ defined by 
$$\sum_{i=1}^5 x_i=\sum_{i=1}^5 x_i^3 =0,$$
where $(x_1:\cdots :x_5)$ is homogeneous coordinates of ${\bf P}^4$.  
Its Hessian surface $H$ is given by
\begin{equation}\label{Hessian}
\sum_{i=1}^5 x_i=\sum_{i=1}^5 {1\over x_i} =0
\end{equation}
which is invariant under the Cremona transformation $c : (x_i)\to (1/x_i)$.
The Cremona transformation induces a fixed point free involution $\sigma$ of the minimal resolution $X$ of the Hessian quartic surface $H$ and 
the quotient surface $Y=X/\la\sigma\ra$ is the Enriques surface of type VI.
The surface $H$ has 10 lines $\ell_{ij}$ defined by $x_i=x_j=0$ and 10 nodes $p_{ijk}$ defined by
$x_i=x_j=x_k=0$ which are interchanged by $c$ as $c(p_{ijk}) = \ell_{mn}$ where
$\{i,j,k,l,m\}=\{1,2,3,4,5\}$.  Thus we have 10 $(-2)$-curves on $Y$ which are the images of ten exceptional curves (or ten lines).  The dual graph of 10 $(-2)$-curves is the Petersen graph (see Figure \ref{typeVI}, (A)).
The speciality of the Clebsch diagonal cubic is that it has ten Eckardt
points.  The hyperplane section $x_i+x_j=0$ of $H$ is $2\ell_{ij}$ plus 2 additional lines interchanged
by $c$.  Thus we have additionally ten $(-2)$-curves on $Y$ forming the dual graph
given in Figure \ref{typeVI}, (B). The dual graph of the obtained 20 $(-2)$-curves on $Y$ is of type $\Gamma_{\rm VI}$.

\begin{figure}[htbp]
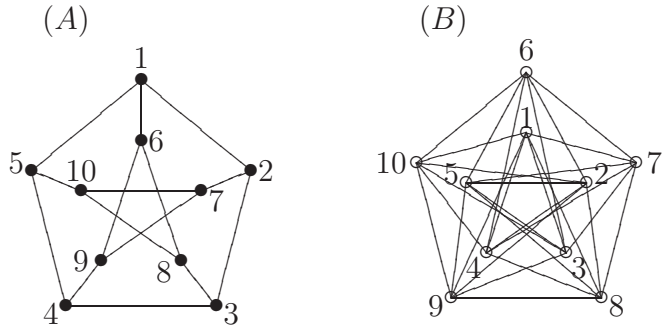

\begin{center}
$
\begin{array}{ccc}
 \hspace{-2cm} (A) &  \hspace{-1cm} (B) \\
\xy
@={(-10,-2),(10,-2),(-14.6,16),(14.6,16),(00,28),(-5.3,4),(5.3,4),(-8,13.3),(8,13.3),(00,20)}@@{*{\bullet}};
(-10,-2)*{};(10,-2)*{}**\dir{-};
(-10,-2)*{};(-14.6,16)*{}**\dir{-};
(-14.6,16)*{};(00,28)*{}**\dir{-};
(00,28)*{};(14.6,16)*{}**\dir{-};
(14.6,16)*{};(10,-2)*{}**\dir{-};
(-5.3,4)*{};(-10,-2)*{}**\dir{-};
(-5.3,4)*{};(8,13.3)*{}**\dir{-};
(-5.3,4)*{};(00,20)*{}**\dir{-};
(5.3,4)*{};(10,-2)*{}**\dir{-};
(5.3,4)*{};(-8,13.3)*{}**\dir{-};
(5.3,4)*{};(00,20)*{}**\dir{-};
(-14.6,16)*{};(-8,13.3)*{}**\dir{-};
(8,13.3)*{};(-8,13.3)*{}**\dir{-};
(8,13.3)*{};(14.6,16)*{}**\dir{-};
(00,20)*{};(00,28)*{}**\dir{-};
(00,31)*{1};(02,20)*{6};(16.6,16)*{2};(10,12)*{7};(3,2.6)*{8};(12,-3)*{3};(-12,-3)*{4};
(-8,4)*{9};(-16.6,17)*{5};(-8,16)*{10};
\endxy 

& 
 \hspace{1cm}

\xy
@={(-10,-2),(10,-2),(-14.6,16),(14.6,16),(00,28),(-5.3,4),(5.3,4),(-8,13.3),(8,13.3),(00,20)}@@{*{\circ}};
(-5.3,4)*{};(8,13.3)*{}**\dir{-};
(-5.3,4)*{};(00,20)*{}**\dir{-};
(5.3,4)*{};(-8,13.3)*{}**\dir{-};
(5.3,4)*{};(00,20)*{}**\dir{-};
(8,13.3)*{};(-8,13.3)*{}**\dir{-};
(-10,-2)*{};(10,-2)*{}**\dir{-};
(-10,-2)*{};(-14.6,16)*{}**\dir{-};
(-14.6,16)*{};(00,28)*{}**\dir{-};
(00,28)*{};(14.6,16)*{}**\dir{-};
(14.6,16)*{};(10,-2)*{}**\dir{-};
(-10,-2)*{};(-8,13.3)*{}**\dir{-};
(-10,-2)*{};(8,13.3)*{}**\dir{-};
(-10,-2)*{};(5.3,4)*{}**\dir{-};
(-10,-2)*{};(00,20)*{}**\dir{-};
(-14.6,16)*{};(-5.3,4)*{}**\dir{-};
(-14.6,16)*{};(8,13.3)*{}**\dir{-};
(-14.6,16)*{};(5.3,4)*{}**\dir{-};
(-14.6,16)*{};(00,20)*{}**\dir{-};
(10,-2)*{};(-8,13.3)*{}**\dir{-};
(10,-2)*{};(8,13.3)*{}**\dir{-};
(10,-2)*{};(-5.3,4)*{}**\dir{-};
(10,-2)*{};(00,20)*{}**\dir{-};
(14.6,16)*{};(-5.3,4)*{}**\dir{-};
(14.6,16)*{};(-8,13.3)*{}**\dir{-};
(14.6,16)*{};(5.3,4)*{}**\dir{-};
(14.6,16)*{};(00,20)*{}**\dir{-};
(00,28)*{};(-5.3,4)*{}**\dir{-};
(00,28)*{};(-8,13.3)*{}**\dir{-};
(00,28)*{};(5.3,4)*{}**\dir{-};
(00,28)*{};(8,13.3)*{}**\dir{-};
(00,31)*{6};(00,22)*{1};(17,16)*{7};(10,14.5)*{2};(7,2)*{3};(12,-3)*{8};(-12,-3)*{9};
(-7,2)*{4};(-18,16)*{10};(-10,14.5)*{5};
\endxy 
\end{array}
$
 \caption{The dual graph $\Gamma_{\rm VI}$ of Type VI}
 \label{typeVI}
\end{center}
\end{figure}

In characteristic 3, the surface $H$ defined by the equation (\ref{Hessian}) still exists, but the Cremona transformation $c$ has five fixed points  
$(-1:1:1:1:1), (1:-1:1:1:1), (1:1:-1:1:1), (1:1:1:-1:1), (1:1:1:1:-1)$ on $H$.
Also the hyperplane section $x_i+x_j=0$ of $H$ is $2\ell_{ij}$ plus a double line which is invariant under $c$.  Thus the induced involution $\sigma$ of $X$ has five fixed curves and ten invariant $(-2)$-curves.  The quotient surface $S=X/\la \sigma\ra$ is a Coble surface with five boundary components $B_1,\ldots, B_5$ and contains 10 $(-1)$-curves.
The surface $S$ has 10 $(-2)$-curves (the images of 10 nodes) whose dual graph is the one in Figure
\ref{typeVI}, (A).  The dual graph of 
ten $(-1)$-roots ${1\over 2}B_i + {1\over 2}B_j +2E$ is given in Figure \ref{typeVI}, (B).
We denote by $e_i$ (resp. $f_i$) the root corresponding to the vertex $i$ in the Figure (A)
(resp. Figure (B)).  Then $e_i \cdot f_j = 2\delta_{ij}$ where $\delta_{ij}$ is the Kronecker delta.
The dual graph satisfies the condition in Theorem \ref{Vinberg} and hence
${\rm Aut}(S)$ is finite (Proposition \ref{FiniteIndex}).  
Moreover it follows from Proposition \ref{Vinberg2} that
$S$ has exactly twenty effective irreducible roots.
Obviously the symmetric group $S_5$ of degree 5 acts on $H$ commuting with 
$c$ and hence on $S$ as automorphisms.

\begin{remark}\label{VIremark}
The surface $S$ is obtained from the rational elliptic surface with singular fibers
of type ${\rm I}_6, {\rm I}_3, {\rm III}$ with a section.  
For example, consider an elliptic fibration with singular fibers of type 
$\tilde{A}_5+\tilde{A}_2+2\tilde{A}_1$ defined by 
$$|e_1+e_2+e_3+e_4+e_9+e_6| =|f_{5}+f_{7}+f_{8}| = |2(e_{10}+f_{10})|.$$
Contracting three $(-1)$-curves in $\tilde{A}_2$ and one $(-1)$-curve in $\tilde{A}_1$
and then contracting the image of the $(-2)$-curve $e_{10}$, we get a minimal elliptic fibration
$\pi$ with singular fibers of type ${\rm I}_6, {\rm I}_3, {\rm III}$.  
We remark that such a Jacobian rational elliptic fibration in characteristic 3 is unique
(Ito \cite{Ito}).  
\end{remark}

\begin{proposition}\label{typeV3prop}{\rm (Dolgachev and Kond\=o \cite[Theorem 9.8.10]{DK})}
Let $S$ be the Coble surface of type {\rm VI} with five boundary components in characteristic $3$.
The surface $S$ has the automorphism group ${\rm Aut}(S)$ isomorphic to $S_5$.  
The $R$-invariant of $S$ is $(E_6\oplus D_5, {\bf Z}/2{\bf Z})$.
\end{proposition}

\end{example}

\begin{example}\label{TypeVI5} {\bf A Coble surface of type VI in characteristic 5 with one boundary component.}

We use the same notation as in Example \ref{TypeVI3}.  In characteristic 5, the Cremona transformation $c$ has a fixed point $q=(1:1:1:1:1)$.  The quotient surface $X/\la\sigma\ra$ is a Coble surface $S$ 
 with one boundary component.  
Since the lines $\ell_{ij}$ and the hyperplane sections $x_i+x_j=0$ 
do not pass through $q$, $S$ contains 20 $(-2)$-curves forming the dual graph $\Gamma_{\rm VI}$.

\begin{proposition}\label{typeVI5prop}{\rm (Dolgachev and Kond\=o \cite[Theorem 9.8.13]{DK})}
Let $S$ be the Coble surface of type {\rm VI} with one boundary component in characteristic $5$.
The surface $S$ has the automorphism group ${\rm Aut}(S)$ isomorphic to $S_5$.  
The $R$-invariant of $S$ is $(E_6\oplus A_4, \{ 0 \})$.
\end{proposition}

\begin{remark}
We use the same notation as in Example \ref{TypeI2}.  In Kond\=o \cite[Example VI]{Ko}, the author
gave an Enriques surface $Y$ of type VI whose 
covering $K3$ surface $X$ is the minimal resolution of the double covering of ${\bf P}^1\times {\bf P}^1$ branched along the divisor $E + C_+ + C_-$ of bi-degree $(4,4)$, 
where $E, C_+, C_-$ are defined by
$$E: u_0^2v_0^2 + u_1^2v_1^2 - u_0^2v_1^2 + {5\over 3}u_1^2v_0^2 + {8\over 3}u_0u_1v_0v_1=0;$$
$$C_+: 4u_1v_0 = (u_0+u_1)(v_0-v_1);\quad C_-: -4u_1v_0 = (u_0-u_1)(v_0+v_1).$$

In characteristic 5, we replace $E$ by the curve defined by
$$u_0^2v_0^2 + u_1^2v_1^2 - u_0^2v_1^2 + u_0u_1v_0v_1=0$$
which has a node at $(0:1,1:0)$ fixed by the involution $\phi$.  
The obtained surface is a Coble surface of type VI with one boundary.
\end{remark}

\end{example}

\begin{example}\label{TypeVII} {\bf A Coble surface of type VII in characteristic 5 with one boundary component.}

The dual graph $\Gamma_{\rm VII}$ of the Enriques surface of type VII is given by the following
Figure \ref{VII}.  We remark that this diagram is the same as the one in Kond\=o \cite[Fig.7.7]{Ko}.

\begin{figure}[!htb]
 \begin{center}
  \includegraphics[width=50mm]{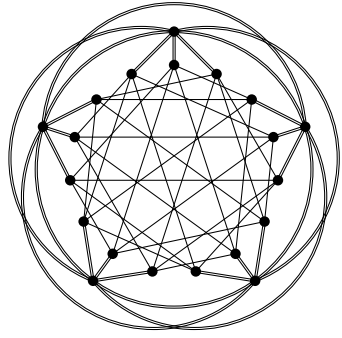}
 \end{center}
 \caption{The dual graph $\Gamma_{\rm VII}$ of Type VII}
 \label{VII}
\end{figure}

Ohashi \cite{MO1} found the following construction of the Enriques surface of type VII in 
characteristic 0 as an analogue of the Hessian quartic surface.
Consider the surface $X'$ in ${\bf P}^4$ defined by
$$\sum_{i<j} x_ix_j=\sum_{i<j} {1\over x_ix_j} =0$$
which has 5 nodes at $(1:0:0:0:0),\ldots, (0:0:0:0:1)$ and contains five curves of arithmetic genus 4 and with 4 nodes defined by $x_i=0$.  The Cremona transformation $c: (x_i)\to (1/x_i)$ changes
5 nodes and 5 curves.  Moreover $X'$ contains the two lines $\ell_{\pm}$
passing through the two points $(0:0:0:0:1)$ and $(1:-1:\pm \sqrt{-1}:\mp\sqrt{-1}:0)$.  
The Cremona transformation $c$ interchanges $\ell_+$ and $\ell_-$.  
By changing the coordinates $(x_1:\cdots : x_5)$ 
by the action of fifteen involutions $(ij)(kl)$ in $S_5$, we obtain
30 lines on $X'$ as the images of $\ell_{\pm}$.  
The pre-images of these 40 curves on the minimal resolution $X$ of $X'$ are 
40 $(-2)$-curves.  The induced involution $\sigma$ of $X$ is fixed point free and 
the quotient surface is the
Enriques surface of type VII containing 20 $(-2)$-curves forming the dual graph 
$\Gamma_{\rm VII}$.  This construction works well in any characteristic $p\ne 2,5$.

Now assume that the ground field $k$ is an algebraically closed field in characteristic 5.
Then $X'$ has additionally a node $(1:1:1:1:1)$ which is fixed by $c$.
Thus the induced involution $\sigma$ has a fixed curve on the minimal resolution $X$ of $X'$.
The quotient surface $X/\la\sigma\ra$ is a Coble surface $S$ with one boundary component. The above 20 $(-2)$-curves exist in this case, too.  
The dual graph satisfies the condition in Theorem \ref{Vinberg} and hence
${\rm Aut}(S)$ is finite (Proposition \ref{FiniteIndex}).  Moreover it follows from Proposition \ref{Vinberg2} that $S$ has exactly twenty effective irreducible roots.

\begin{proposition}\label{typeVIIprop}{\rm (Dolgachev and Kond\=o \cite[Theorem 9.8.15]{DK})}
Let $S$ be the Coble surface of type {\rm VII} with one boundary component in characteristic $5$.
The surface $S$ has the automorphism group ${\rm Aut}(S)$ isomorphic to $S_5$.  
The $R$-invariant of $S$ is $(A_9\oplus A_1, {\bf Z}/2{\bf Z})$.
\end{proposition}

\end{example}

\begin{example}\label{TypeMI} {\bf A Coble surface of type MI in characteristic 3 with 2 boundary components.}

This surface was suggested by Mukai and Ohashi \cite{MO2}.  
First we give a definition of the dual graph of type $\Gamma_{\rm MI}$.
We denote by $(ij)$ the transposition of $i$ and $j$ 
($1 \leq i < j \leq 6$) which is classically called Sylvester's duad.
Six letters $1,2,3,4,5,6$ can be arranged in three pairs of duads,
for example, $(12, 34, 56)$, called Sylvester's syntheme
(it is understood that $(12, 34, 56)$ is the same as $(12, 56, 34)$ or $(34, 12, 56)$).
Duads and synthemes are in $(3,3)$ correspondence, that is, each syntheme consists
of three duads and each duad belongs to three synthemes (see \cite[page 4]{ATLAS}). 
A triad $(ijk)$ is a subset of $\{1,2,3,4,5,6\}$ of cardinality 3, up to complementary set.  There are ten triads.

In the following we use the notation  
$\{i,j,k,l,m,n\}=\{a,b,c,d,e,f\} = \{1,2,\ldots, 6\}$.
Let $\Gamma_{\rm MI}$ be a graph with 40 vertices indexed by duads $(ij)$, synthemes 
$(ij,kl,mn)$ and triads $(ijk)$.  We denote by $v_{ij}, \ v_{ij,kl},\ v_{ijk}$ the vertex indexed by 
$(ij)$, $(ij,kl,mn)$, $(ijk)$ respectively.
Two vertices $v_{ij}$ and $v_{kl}$ are joined by a single edge if and only if $|\{ i, j\}\cap \{ k,l\}|=1$.  Two vertices $v_{ij,kl}$ and $v_{ab, cd}$ are joined by a single edge 
if and only if the two synthemes $(ij,kl,mn)$ and $(ab,cd,ef)$ have no common duads. 
The graph $\Gamma_{\rm MI}^d$ of 15 vertices $v_{ij}$ and the one 
$\Gamma_{\rm MI}^s$ of 15 vertices $v_{ij,kl}$
are isomorphic and are given by the dual of the following graph with 6 vertices and 15 edges in Figure \ref{typeMI}.

\begin{figure}[htbp]
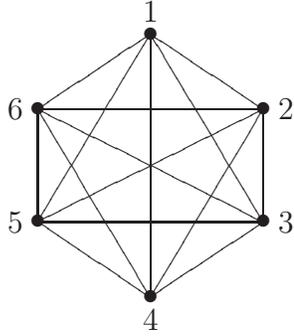

\begin{center}
\xy
(-50,0)*{};
@={(00,0),(00,35),(15,10),(-15,10),(15,25),(-15,25)}@@{*{\bullet}};
(00,0)*{};(00,35)*{}**\dir{-};
(00,0)*{};(15,10)*{}**\dir{-};
(00,0)*{};(-15,10)*{}**\dir{-};
(00,0)*{};(15,25)*{}**\dir{-};
(00,0)*{};(-15,25)*{}**\dir{-};
(00,35)*{};(15,10)*{}**\dir{-};
(00,35)*{};(-15,10)*{}**\dir{-};
(00,35)*{};(15,25)*{}**\dir{-};
(00,35)*{};(-15,25)*{}**\dir{-};
(15,10)*{};(-15,10)*{}**\dir{-};
(15,10)*{};(15,25)*{}**\dir{-};
(15,10)*{};(-15,25)*{}**\dir{-};
(-15,10)*{};(15,25)*{}**\dir{-};
(-15,10)*{};(-15,25)*{}**\dir{-};
(-15,25)*{};(15,25)*{}**\dir{-};
(0,-3)*{4};(0,38)*{1};(18,25)*{2};(18,10)*{3};(-18,10)*{5};(-18,25)*{6};
\endxy 
 \caption{The dual of $\Gamma_{\rm MI}^d$ and $\Gamma_{\rm MI}^s$}
 \label{typeMI}
\end{center}
\end{figure}

\noindent
Two vertices $v_{ij}$ and $v_{ab,cd}$ are joined by a double edge if and only if
the duad $(ij)$ is contained in the syntheme $(ab,cd,ef)$.  
We define a complete graph $\Gamma_{\rm MI}^t$ with the ten vertices $v_{ijk}$ indexed by triads and with double edges.
A vertex $v_{ij}$ is joined with four vertices $v_{ijk}, v_{ijl}, v_{ijm}, v_{ijn}$ by a double edge.  
A vertex $v_{ij,kl}$ is joined with four vertices $v_{ikm}, v_{ikn}, v_{ilm}, v_{iln}$ by a double edge.
Thus a vertex $v_{ijk}$ is joined with 6 vertices $v_{ij}$, $v_{ik}$, $v_{jk}$, $v_{lm}$,
$v_{ln}$, $v_{mn}$ by a double edge and with 6 vertices $v_{il,jm}$, $v_{il,jn}$, 
$v_{im,jl}$, $v_{im,jn}$, $v_{in,jl}$, $v_{in,jm}$ by a double edge.  
The obtained graph combining $\Gamma_{\rm MI}^d, \Gamma_{\rm MI}^s, \Gamma_{\rm MI}^t$ is denoted by
$\Gamma_{\rm MI}$.  Obviously the symmetric group $S_6$ of degree $6$ acts on $\Gamma_{\rm MI}$.  The automorphism group of $\Gamma_{\rm MI}$ is isomorphic to ${\rm Aut}(S_6)\cong
S_6\cdot {\bf Z}/2{\bf Z}$ where ${\bf Z}/2{\bf Z}$ is generated by an outer automorphism
of $S_6$ changing duads and synthemes.

Now we give a Coble surface in characteristic 3 with two boundary components and with 40 
irreducible effective roots forming the dual graph $\Gamma_{\rm MI}$. 
The following was given in Katsura, Kondo \cite[\S 4]{KatsuraKondo}. 
Consider the smooth quadric $Q={\bf P}^1({\bf F}_9)\times {\bf P}^1({\bf F}_9)$
defined over the finite field ${\bf F}_9$.  
Let $(u_0:u_1,v_0:v_1)$ be bi-homogeneous coordinates of $Q$ and let $B, B'$ be 
the non-singular rational curves in $Q$ defined by 
$$B: u_0v_0^3=u_1v_1^3, \quad B': u_0^3v_0= u_1^3v_1.$$
Let $\zeta=1-\sqrt{-1}$ be a primitive 8th root of unity.  
The curves $B$ and $B'$ meet at ten points
$$p_1=(1:\zeta, 1:-\zeta),\ p_2=(1:\zeta^2, 1:\zeta^{2}),\ p_3=(1:\zeta^3, 1:-\zeta^3),$$  
$$p_4=(1:-1,1:-1),\ p_5=(1:-\zeta,1:\zeta),\ p_6=(1:-\zeta^2, 1:-\zeta^2),$$
$$p_7=(1:-\zeta^3,1:\zeta^3),\ p_8=(1:1,1:1),\ p_9=(1:0,0:1), \ p_{10}=(0:1,1:0).$$

\noindent
There are 30 curves of bidegree $(1,1)$ on $Q$ which are identified with the duads and synthemes in the following way.  In the bracket we indicate the 
four points among the ten points $p_1,\ldots, p_{10}$ through which the curve passes.

\medskip

$(12): u_0v_0-u_1v_1=0\quad (p_4, p_8, p_9, p_{10}),$

$(13): u_0v_0-\zeta^3(u_0v_1+u_1v_0)+\zeta^2u_1v_1=0\quad (p_2,p_3,p_4,p_7),$

$(14): u_0v_0+\zeta(u_0v_1+u_1v_0)-\zeta^2u_1v_1=0\quad (p_1,p_2,p_5,p_8),$

$(15): u_0v_0+\zeta^2u_1v_0-u_1v_1=0\quad (p_3, p_5, p_6, p_9),$

$(16): u_0v_0+\zeta^2u_0v_1-u_1v_1=0\quad (p_1, p_6, p_7, p_{10}),$

$(23): u_0v_0-\zeta(u_0v_1+u_1v_0)-\zeta^2u_1v_1=0\quad (p_1,p_4,p_5,p_6),$

$(24): u_0v_0+\zeta^3(u_0v_1+u_1v_0)+\zeta^2u_1v_1=0\quad (p_3,p_6,p_7,p_8),$

$(25): u_0v_0-\zeta^2u_1v_0-u_1v_1=0\quad (p_1, p_2, p_7, p_9),$

$(26): u_0v_0-\zeta^2u_0v_1-u_1v_1=0\quad (p_2, p_3, p_5, p_{10}),$

$(34): u_0v_0+u_1v_1=0\quad (p_2, p_6, p_9, p_{10}),$

$(35): u_0v_0+u_0v_1+u_1v_1=0\quad (p_1, p_3, p_8, p_{10}),$

$(36): u_0v_0+u_1v_0+u_1v_1=0\quad (p_5, p_7, p_8, p_9),$

$(45): u_0v_0-u_0v_1+u_1v_1=0\quad (p_4, p_5, p_7, p_{10}),$

$(46): u_0v_0-u_1v_0+u_1v_1=0\quad (p_1, p_3, p_4, p_9),$

$(56): u_0v_1-u_1v_0=0\quad (p_2, p_4, p_6, p_8),$

$(12,34,56): u_0v_1+u_1v_0=0\quad (p_1, p_3, p_5, p_7),$

$(12,35,46): u_0v_0+u_0v_1-u_1v_0+u_1v_1=0\quad (p_2,p_5,p_6,p_7),$

$(12,36,45): u_0v_0-u_0v_1+u_1v_0+u_1v_1=0\quad (p_1,p_2,p_3,p_6),$

$(13,24,56): u_0v_0-\zeta^2u_1v_1=0\quad (p_1, p_5, p_9, p_{10}),$

$(13,25,46): u_0v_0-\zeta^3u_0v_1 + \zeta^2u_1v_1=0\quad (p_5, p_6, p_8, p_{10}),$

$(13,26,45): u_0v_0-\zeta^3u_1v_0 + \zeta^2u_1v_1=0\quad (p_1, p_6, p_8, p_9),$

$(14,23,56): u_0v_0+\zeta^2u_1v_1=0\quad (p_3, p_7, p_9, p_{10}),$

$(14,25,36): u_0v_0+\zeta u_0v_1 - \zeta^2u_1v_1=0\quad (p_3, p_4, p_6, p_{10}),$

$(14,26,35): u_0v_0+\zeta u_1v_0 - \zeta^2u_1v_1=0\quad (p_4, p_6, p_7, p_9),$

$(15,23,46): u_0v_0-\zeta u_0v_1 - \zeta^2u_1v_1=0\quad (p_2, p_7, p_8, p_{10}),$

$(15,24,36): u_0v_0+\zeta^3u_0v_1 + \zeta^2u_1v_1=0\quad (p_1, p_2, p_4, p_{10}),$

$(15,26,34): u_0v_0-\zeta^2(u_0v_1-u_1v_0)-u_1v_1=0\quad (p_1, p_4, p_7, p_8),$

$(16,23,45): u_0v_0-\zeta u_1v_0 - \zeta^2u_1v_1=0\quad (p_2, p_3, p_8, p_9),$

$(16,24,35): u_0v_0+\zeta^3u_1v_0 + \zeta^2u_1v_1=0\quad (p_2, p_4, p_5, p_9),$

$(16,25,34): u_0v_0+\zeta^2(u_0v_1-u_1v_0)-u_1v_1=0\quad (p_3, p_4, p_5, p_8).$

\medskip

Let $S$ be the surface obtained by blowing up ten points $p_1,\ldots, p_{10}$ of $Q$ and let $E_i'$ be the exceptional curve over the point $p_i$. 
We denote by the same symbols the proper transforms of the above curves.
Then $B^2=B'^2=-4$ and the 30 curves above are $(-2)$-curves.
The surface $S$ is a Coble surface with two boundary components $B, B'$.

Now we define ten $(-1)$-roots by
$$E_i= {1\over 2}B +{1\over 2}B' + 2E_i' \quad (1\leq i \leq 10).$$
The dual graph of $E_1,..., E_{10}$ is the complete graph with double edges.
We can identify 10 $(-1)$-roots $E_i$ with ten triads $(ijk)$ as follows:
$$(123): E_4,\ \ (124): E_8,\ \ (125): E_9,\ \ (126): E_{10}, \ \ 
(134): E_2,$$
$$(135): E_3,\ \ (136): E_7,\ \ (145): E_5, \ \ (146): E_{1},
\ \ (156): E_6.$$

Denote by $v_{ij}$, $v_{ij,kl}$, $v_{ijk}$ the classes of
thirty $(-2)$-curves and ten $(-1)$-roots $E_i$ in ${\rm CM}(S)$
by using the above identification.
It is now easy to see that the dual graph of 40 irreducible effective roots 
$\{ v_{ij}, v_{ij,kl}, v_{ijk}\}$ is $\Gamma_{\rm MI}$.

\begin{lemma}\label{parabolicMI}
Every connected parabolic subdiagram of $\Gamma_{\rm MI}$ is a component of the following maximal parabolic subdiagrams of $\Gamma_{\rm MI}:$
$$(1)\ \tilde{A}_5+2\tilde{A}_2+2\tilde{A}_1,\quad (2)\ \tilde{A}_4+\tilde{A}_4,\quad (3)\ \tilde{A}_3+\tilde{A}_3+2\tilde{A}_1 +\tilde{A}_1,\quad (4)\ \tilde{A}_2+\tilde{A}_2+\tilde{A}_2 + \tilde{A}_2.$$
Every maximal parabolic diagram defines a special genus one fibration. 
\end{lemma}
\begin{proof} 
The assertion follows from a direct calculation.  We give an example of each type.  
The last vertex is a bi-section of the genus one fibration defined by this parabolic subdiagram.

\noindent
$(1) \ \la v_{ij}, v_{jk}, v_{kl}, v_{lm}, v_{mn}, v_{ni}\ra, \ 
\la v_{im,jl}, v_{il, jn}, v_{ik,jm}\ra, \ \la v_{il,jm}, v_{ikm}\ra, \ v_{ij,km};$

\noindent
$(2) \  \la v_{ij}, v_{jk}, v_{kl}, v_{lm}, v_{mi}\ra,\ \la v_{ik,jm}, v_{il, jn}, v_{ik,jl}, v_{il,jm}, v_{in,jl}\ra,\ v_{ln};$

\noindent
$(3) \  \la v_{km}, v_{kn}, v_{ln}, v_{lm}\ra,\ \la v_{ik,jl}, v_{in, jm}, v_{il,jk}, v_{im,jn}\ra,\ 
 \la v_{ij}, v_{ij,kl}\ra,\ \la v_{ikl}, v_{imn}\ra, \ v_{ik};$

\noindent
$(4) \  \la v_{ij}, v_{jk}, v_{ki}\ra,\ \la v_{lm}, v_{mn}, v_{ln}\ra,\ 
\la v_{im,jn}, v_{in, jl}, v_{il,jm}\ra,\ \la v_{in,jm}, v_{il,jn}, v_{im,jl}\ra, \ v_{ij,kl}.$
\end{proof}

Since the 40 $(-2)$-classes are effective, by Lemma \ref{parabolicMI}, 
Theorem \ref{Vinberg} and Proposition \ref{FiniteIndex}, 
${\rm Aut}(S)$ is finite. Moreover it follows from Proposition \ref{Vinberg2} that $S$ has exactly forty effective irreducible roots. 
The group ${\rm Aut}(S)$ is isomorphic to the subgroup of ${\rm Aut}(Q)$ preserving $B+B'$.
This group is generated by an involution
$(u_0:u_1, v_0:v_1) \to (v_0:v_1,u_0:u_1)$ and
${\rm PGL}_2({\bf F}_9)$ acting on $Q={\bf P}^1\times {\bf P}^1$ by
$$
\left(
  \begin{array}{cc}
      a & b \\
      c & d     
    \end{array}
  \right)
: (u_0:u_1,v_0:v_1) \to (au_0+bu_1:cu_0+du_1, d^3v_0+c^3v_1:b^3v_0+a^3v_1).$$
Thus ${\rm Aut}(S)$ acts on $\Gamma_{\rm MI}$ faithfully and hence
${\rm Aut}(S) \subset {\rm Aut}(\Gamma_{\rm MI})$.
On the other hand, ${\rm Aut}(\Gamma_{\rm MI})$ is isomorphic to ${\rm Aut}(S_6)$.
Both groups have the same order 1440 and hence ${\rm Aut}(S)\cong {\rm Aut}(S_6)$.

Let
$L_i: u_1=\zeta^iu_0,\ L_i': v_1= \zeta^iv_0 \ (1\leq i\leq 8),\ L_9: u_1=0,\ L_{10}: u_0=0,\
L'_9: v_1=0,\ L'_{10}: v_0=0.$
Then $L_i$ (resp. $L_i'$) meets $B$ (resp. $B'$) at one of the above ten points with multiplicity 3.
By taking the double cover of $S$ branched along $B+B'$ we obtain a non-singular surface
which is nothing but the Fermat quartic surface $X$.  It is well known that the Fermat quartic surface
is a supersingular $K3$ surface with Artin invariant 1 (e.g. Shioda \cite{Shi}).
The pre-images of $B, B'$ are two skew lines and
those of $E_i'$ are ten lines meeting the two skew lines.
The pre-image of $L_i$ (resp. $L_i'$) splits into two $(-2)$-curves and each of the above 30 $(-2)$-curves splits into disjoint two $(-2)$-curves.  Thus we have $112\ (=2 + 10+40 + 60)$ $(-2)$-curves on $X$ which are lines on $X$.  

Now we calculate the $R$-invariant $(K,H)$.
There exists a Dynkin diagram of type $A_5$ in $\Gamma^d_{\rm MI}$, for example, the vertices indexed by $12, 23, 34, 45, 56$.  
Similarly there exists a Dynkin diagram of type $A_5$ in $\Gamma^s_{\rm MI}$.
Thus $K$ contains $A_5^{\oplus 2}\oplus A_1^{\oplus 2}$ where $A_1^{\oplus 2}$ 
is generated by $\tilde{B}+\tilde{B}', \tilde{B}- \tilde{B}'$.  
If $K\ne A_5^{\oplus 2}\oplus A_1^{\oplus 2}$, then
$K=E_6^{\oplus 2}$ because ${\rm det}(K)$ is divisible by $3^2$ (Proposition \ref{SSK3Pic}, 
Lemma \ref{overlattice}). However
$q_{E_6/2E_6}$ is non-singular and hence $H=\{0\}$. 
This contradicts the fact that ${\rm CM}(S)$ has rank 10. Therefore 
$K=A_5^{\oplus 2}\oplus A_1^{\oplus 2}$. 
Since ${\rm CM}(S)\cong E_{10}$ (Proposition
\ref{CME10}), ${\rm dim}\ |H|\geq 3$.  Since the nullity of $q_{K/2K}$ is 3,
we have $H\cong ({\bf Z}/2{\bf Z})^3$.  
We now have the following.

\begin{proposition}\label{CobleMI} {\rm (Dolgachev and Kond\=o \cite[Theorem 10.5.18]{DK})}
The surface $S$ is a Coble surface of type {\rm MI} in characteristic $3$ with two boundary components
whose covering $K3$ surface is the Fermat quartic surface.  It has forty effective irreducible roots forming the dual graph $\Gamma_{\rm MI}$. 
The automorphism group ${\rm Aut}(S)$ is isomorphic to 
${\rm Aut}(S_6)$ and the $R$-invariant $(K,H)$ is $(A_5^{\oplus 2}\oplus A_1^{\oplus 2}, ({\bf Z}/2{\bf Z})^3)$.
\end{proposition} 

\end{example}

\begin{example}\label{TypeMII} {\bf A Coble surface of type MII in characteristic 3 with 8 boundary components.}

This surface was also suggested by Mukai and Ohashi \cite{MO2}.  First we explain the dual graph 
$\Gamma_{\rm MII}$ with 40 vertices.
The graph $\Gamma_{\rm MII}$ is divided into three subgraphs $\Gamma_{0}, \Gamma_{1}, \Gamma_{2}$
such that $\Gamma_{0}$ (resp.  $\Gamma_{1}$) has sixteen vertices (resp. twelve vertices) and 
$\Gamma_{2}$ is a copy of $\Gamma_{1}$.  
We write sixteen vertices of $\Gamma_0$ by $v_{ij}$ $(1\leq i,j\leq 4)$ and join two distinct vertices
$v_{ij}$ and $v_{kl}$ by a single edge if and only if $i=k$ or $j=l$.
The graph $\Gamma_0$ is the dual of the following graph of eight vertices and sixteen edges given in Figure \ref{typeMII}:

\begin{figure}[htbp]
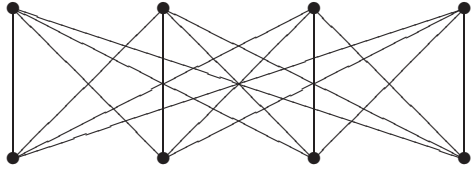

\xy
(-30,25)*{};
@={(00,20),(20,20),(40,20),(60,20),(00,0),(20,0),(40,0),(60,0)}@@{*{\bullet}};
(00,0)*{};(00,20)*{}**\dir{-};
(00,0)*{};(20,20)*{}**\dir{-};
(00,0)*{};(40,20)*{}**\dir{-};
(00,0)*{};(60,20)*{}**\dir{-};
(20,0)*{};(00,20)*{}**\dir{-};
(20,0)*{};(20,20)*{}**\dir{-};
(20,0)*{};(40,20)*{}**\dir{-};
(20,0)*{};(60,20)*{}**\dir{-};
(40,0)*{};(00,20)*{}**\dir{-};
(40,0)*{};(20,20)*{}**\dir{-};
(40,0)*{};(40,20)*{}**\dir{-};
(40,0)*{};(60,20)*{}**\dir{-};
(60,0)*{};(00,20)*{}**\dir{-};
(60,0)*{};(20,20)*{}**\dir{-};
(60,0)*{};(40,20)*{}**\dir{-};
(60,0)*{};(60,20)*{}**\dir{-};
\endxy 
 \caption{the dual of $\Gamma_0$}
 \label{typeMII}
\end{figure}

Let $A_4$, $S_4$ be the alternating group and the symmetric group of degree 4.
The vertices of $\Gamma_1$ and $\Gamma_2$ are indexed by elements of 
$A_4$ and $S_4\setminus
A_4$, respectively,  as follows.  
The graph $\Gamma_1=\{v_{\tau} \ : \ \tau\in A_4\}$
(resp. $\Gamma_2 =\{v_{\tau} \ : \ \tau\in S_4\setminus A_4\}$) is divided into three complete graphs $\Gamma_{11}, \Gamma_{12}, \Gamma_{13}$ with four vertices 
(resp. $\Gamma_{21}, \Gamma_{22}, \Gamma_{23}$).  
Let $K_1$ be the 2-elementary subgroup of $A_4$ of order 4 and
let $K_1, K_2, K_3, K'_1, K'_2, K'_3$ be the complete cosets of $K_1$ in $A_4$ such that
$A_4=K_1\cup K_2\cup K_3$:

$K_{1}=\{ id,\ (12)(34),\ (13)(24),\ (14)(23)\}, 
K_{2}=\{(123),\ (134),\ (142),\ (243) \},$

$K_{3}=\{(124),\ (132),\ (143),\ (234) \},
K'_{1}=\{(12),\ (34),\ (1324),\ (1423) \},$

$K'_{2}=\{(13),\ (24),\ (1234),\ (1432) \}, 
K'_{3}=\{(14),\ (23),\ (1243),\ (1342) \}.$

\noindent
Then
$\Gamma_{1i}= \{v_{\tau} \ : \ \tau\in K_i\}$, $\Gamma_{2j}= \{v_{\tau} \ : \ \tau\in K'_j\}$ $(i, j=1,2,3)$.  
For $\tau \in S_4$, denote by $G(\tau)$ the graph of the map 
$\tau$ on the set $\{1,2,3,4\}$. We join $v_\tau$ and $v_{\tau'}$ by 
$(2-\#(G(\tau)\cap G(\tau')))$-tuple edge.  Thus, for example, each $\Gamma_{ij}$ is a complete graph
with double edge.  A vertex $v_{ij}$ of $\Gamma_0$ and a vertex $v_\tau$ of $\Gamma_1$ or $\Gamma_2$ are joined by a double edge if and only if $(i,j)$ is contained in the graph $G(\tau)$.

Now we give a Coble surface in characteristic 3 with 40 effective irreducible roots 
whose dual graph is isomorphic to $\Gamma_{\rm MII}$.
Let $E$ be a curve defined by $y^2=x^3-x$ which is a supersingular elliptic curve in 
characteristic 3.  Let ${\rm Kum} (E\times E)$ be the Kummer surface associated with the product of 
$E$.  It is known that ${\rm Kum}(E\times E)$ is the supersingular $K3$ surface with 
Artin invariant 1 (e.g. Shioda \cite{Shi}).
The involution $1_E\times (-1_E)$ of $E\times E$ descends to an involution $\sigma$ of 
${\rm Kum}(E\times E)$.
Let $S$ be the quotient surface $X/\la \sigma\ra$.
The quotient of $E\times E$ by the group 
$\la 1_E\times (-1_E), \ (-1_E)\times 1_E\ra (\cong ({\bf Z}/2{\bf Z})^2)$ is isomorphic to 
${\bf P}^1\times {\bf P}^1$ and there is a morphism from $S$ 
to ${\bf P}^1\times {\bf P}^1$ which contracts sixteen $(-1)$-curves on $S$.
Let $p_1=(1:0),\ p_2=(0:1),\ p_3=(1:1),\ p_4=(1:-1)$ be ${\bf F}_3$-rational points of ${\bf P}^1$.
Let $\pi : S \to {\bf P}^1\times {\bf P}^1$ be the blow-up of 
the sixteen ${\bf F}_3$-rational points of 
${\bf P}^1\times {\bf P}^1$, $E_{ij}$ the exceptional curve over $(p_i, p_j)$, and 
$B_i, B'_i$ the proper transforms of $\{p_i\}\times {\bf P}^1, {\bf P}^1\times \{p_i\}$
$(1=1,...,4)$, respectively.  The configuration of $E_{ij}$ and $B_i, B'_i$ is 
as in the following Figure \ref{typeMIIKummer}:

\begin{center}
\begin{figure}[htbp]
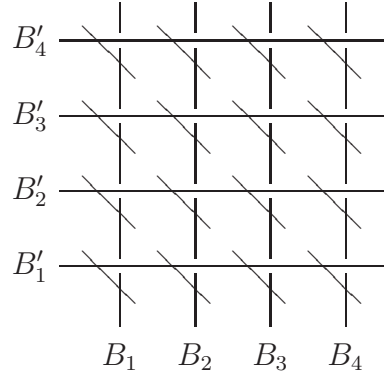

\xy
(-50,0)*{};
(-5,2)*{};(2,-5)*{}**\dir{-};
(-5,12)*{};(2,5)*{}**\dir{-};
(-5,22)*{};(2,15)*{}**\dir{-};
(-5,32)*{};(2,25)*{}**\dir{-};
(5,2)*{};(12,-5)*{}**\dir{-};
(5,12)*{};(12,5)*{}**\dir{-};
(5,22)*{};(12,15)*{}**\dir{-};
(5,32)*{};(12,25)*{}**\dir{-};
(15,2)*{};(22,-5)*{}**\dir{-};
(15,12)*{};(22,5)*{}**\dir{-};
(15,22)*{};(22,15)*{}**\dir{-};
(15,32)*{};(22,25)*{}**\dir{-};
(25,2)*{};(32,-5)*{}**\dir{-};
(25,12)*{};(32,5)*{}**\dir{-};
(25,22)*{};(32,15)*{}**\dir{-};
(25,32)*{};(32,25)*{}**\dir{-};
(00,-8)*{};(00,-1)*{}**\dir{-};
(00,1)*{};(00,9)*{}**\dir{-};
(00,11)*{};(00,19)*{}**\dir{-};
(00,21)*{};(00,29)*{}**\dir{-};
(00,31)*{};(00,35)*{}**\dir{-};
(10,-8)*{};(10,-1)*{}**\dir{-};
(10,1)*{};(10,9)*{}**\dir{-};
(10,11)*{};(10,19)*{}**\dir{-};
(10,21)*{};(10,29)*{}**\dir{-};
(10,31)*{};(10,35)*{}**\dir{-};
(20,-8)*{};(20,-1)*{}**\dir{-};
(20,1)*{};(20,9)*{}**\dir{-};
(20,11)*{};(20,19)*{}**\dir{-};
(20,21)*{};(20,29)*{}**\dir{-};
(20,31)*{};(20,35)*{}**\dir{-};
(30,-8)*{};(30,-1)*{}**\dir{-};
(30,1)*{};(30,9)*{}**\dir{-};
(30,11)*{};(30,19)*{}**\dir{-};
(30,21)*{};(30,29)*{}**\dir{-};
(30,31)*{};(30,35)*{}**\dir{-};
(-8,0)*{};(35,0)*{}**\dir{-};
(-8,10)*{};(35,10)*{}**\dir{-};
(-8,20)*{};(35,20)*{}**\dir{-};
(-8,30)*{};(35,30)*{}**\dir{-};
(0,-12)*{B_1};(10,-12)*{B_2};(20,-12)*{B_3};(30,-12)*{B_4};
(-12,0)*{B'_1};(-12,10)*{B'_2};(-12,20)*{B'_3};(-12,30)*{B'_4};
\endxy 
 \caption{the configuration of $E_{ij}$ and $B_i, B'_i$}
 \label{typeMIIKummer}
\end{figure}
\end{center}

\noindent
Then
$|-2K_{S}|=\left\{\sum_{i=1}^4 (B_i + B'_i)\right\}.$
Thus $S$ is a Coble surface with eight boundary components $\sum_{i=1}^4 (B_i + B'_i)$ and 
isomorphic to ${\rm Kum}(E\times E)/\la \sigma\ra$.  We have sixteen $(-1)$-roots
$$e_{ij}={1\over 2} B_i + {1\over 2} B'_j + 2E_{ij} \quad (1\leq i, j \leq 4).$$
Since $e_{ij}\cdot e_{kl} =1$ if $i=k$ or $j=l$ and $e_{ij}\cdot e_{kl} =0$ for the other 
$\{i,j\}\not= \{k,l\}$, the dual graph of $\{e_{ij}\}$ coincides with $\Gamma_0$.
Thus we identify $\{e_{ij}\}$ with $\{v_{ij}\}$. 
On the other hand, the proper transforms of the following 24 non-singular curves of bidegree $(1,1)$ in 
${\bf P}^1\times {\bf P}^1$ are $(-2)$-curves on $S$.  
We can identify these curves with the 24 vertices
in $\Gamma_1$ and $\Gamma_2$ as follows:  

\medskip
$v_{id}:\ u_0v_1-u_1v_0=0,\quad v_{(12)(34)}:\ u_0v_0+u_1v_1=0,$

$v_{(13)(24)}:\ u_0v_0-u_0v_1-u_1v_0-u_1v_1=0,$

$v_{(14)(23)}:\ u_0v_0+u_0v_1+u_1v_0-u_1v_1=0,$ 

$v_{(142)}:\ u_0v_0+u_0v_1+u_1v_1=0,\quad v_{(123)}:\ u_0v_0-u_1v_0+u_1v_1=0,$

$v_{(134)}:\ u_0v_0-u_0v_1+u_1v_0=0,\quad v_{(243)}:\ u_0v_1-u_1v_0-u_1v_1=0,$

$v_{(132)}:\ u_0v_0-u_0v_1+u_1v_1=0,\quad v_{(143)}:\ u_0v_0+u_0v_1-u_1v_0=0,$

$v_{(124)}:\ u_0v_0+u_1v_0+u_1v_1=0,\quad v_{(234)}:\ u_0v_1-u_1v_0+u_1v_1=0,$

$v_{(12)}:\ u_0v_0-u_1v_1=0,\quad v_{(34)} :\ u_0v_1+u_1v_0=0,$

$v_{(1423)}:\ u_0v_0+u_0v_1-u_1v_0+u_1v_1=0,$

$v_{(1324)}:\ u_0v_0-u_0v_1+u_1v_0+u_1v_1=0,$

$v_{(13)}:\ u_0v_0-u_0v_1-u_1v_0=0,\quad v_{(24)}:\ u_0v_1+u_1v_0+u_1v_1=0,$

$v_{(1432)}:\ u_0v_0+u_0v_1-u_1v_1=0,\quad v_{(1234)}:\ u_0v_0+u_1v_0-u_1v_1=0,$

$v_{(14)}: \ u_0v_0+u_0v_1+u_1v_0=0,\quad v_{(23)}:\ u_0v_1+u_1v_0-u_1v_1=0, $

$v_{(1342)}:\ u_0v_0-u_0v_1-u_1v_1=0,\quad v_{(1243)}:\ u_0v_0-u_1v_0-u_1v_1=0.$

\medskip
\noindent
For example, the vertex $v_{23}$ corresponds to $E_{23}$ the exceptional curve 
over the point $(p_2,p_3)$ and
$v_{(12)}$ corresponds to the conic passing through the four points $(p_1,p_2)$, $(p_2,p_1)$, $(p_3,p_3)$,
$(p_4,p_4)$ as in the following Figure \ref{typeMII2}:

\begin{figure}[htbp]
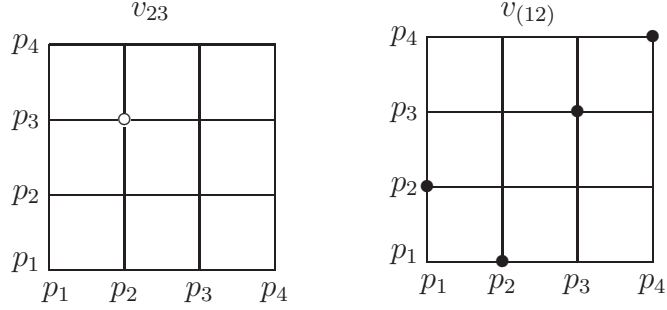

\begin{center}
$
\begin{array}{ccc}
 \hspace{0cm} v_{23} &  \hspace{1cm} v_{(1 2)} \\

\xy
@={(10,20)}@@{*{\circ}};
(00,0)*{};(30,0)*{}**\dir{-};
(00,0)*{};(00,30)*{}**\dir{-};
(10,0)*{};(10,19)*{}**\dir{-};
(10,21)*{};(10,30)*{}**\dir{-};
(20,0)*{};(20,30)*{}**\dir{-};
(30,0)*{};(30,30)*{}**\dir{-};
(00,10)*{};(30,10)*{}**\dir{-};
(00,20)*{};(9,20)*{}**\dir{-};
(11,20)*{};(30,20)*{}**\dir{-};
(00,30)*{};(30,30)*{}**\dir{-};
(1,-3)*{p_1};(10,-3)*{p_2};(20,-3)*{p_3};(30,-3)*{p_4};
(-3,1)*{p_1};(-3,10)*{p_2};(-3,20)*{p_3};(-3,30)*{p_4};
\endxy 

& 
 \hspace{1cm}

\xy
@={(00,10),(10,0),(20,20),(30,30)}@@{*{\bullet}};
(00,0)*{};(30,0)*{}**\dir{-};
(00,0)*{};(00,30)*{}**\dir{-};
(10,0)*{};(10,30)*{}**\dir{-};
(20,0)*{};(20,30)*{}**\dir{-};
(30,0)*{};(30,30)*{}**\dir{-};
(00,10)*{};(30,10)*{}**\dir{-};
(00,20)*{};(30,20)*{}**\dir{-};
(00,30)*{};(30,30)*{}**\dir{-};
(1,-3)*{p_1};(10,-3)*{p_2};(20,-3)*{p_3};(30,-3)*{p_4};
(-3,1)*{p_1};(-3,10)*{p_2};(-3,20)*{p_3};(-3,30)*{p_4};
\endxy

\end{array}
$
 \caption{$v_{23}$ and $v_{(12)}$}
 \label{typeMII2}
\end{center}
\end{figure}

\noindent
Thus we have the decompositions $\Gamma_1=\Gamma_{11}\cup \Gamma_{12}\cup \Gamma_{13}$ and
$\Gamma_2=\Gamma_{21}\cup \Gamma_{22}\cup \Gamma_{23}$:

$\Gamma_{11}=\{v_{id},\ v_{(12)(34)},\ v_{(13)(24)},\ v_{(14)(23)}\}, 
\Gamma_{12}=\{v_{(123)},\ v_{(134)},\ v_{(142)},\ v_{(243)} \},$

$\Gamma_{13}=\{v_{(124)},\ v_{(132)},\ v_{(143)},\ v_{(234)} \},
\Gamma_{21}=\{v_{(12)},\ v_{(34)},\ v_{(1324)},\ v_{(1423)} \},$

$\Gamma_{22}=\{v_{(13)},\ v_{(24)},\ v_{(1234)},\ v_{(1432)} \}, 
\Gamma_{23}=\{v_{(14)},\ v_{(23)},\ v_{(1243)},\ v_{(1342)} \}.$

\noindent

\begin{lemma}\label{parabolicMII}
Every connected parabolic subdiagram of $\Gamma_{\rm MII}$ is a component of the following maximal parabolic subdiagram of $\Gamma_{\rm MII}:$
$$(1)\ \tilde{A}_7+2\tilde{A}_1,\ (2)\ \tilde{A}_5+\tilde{A}_2 + 2\tilde{A}_1,\ (3)\ \tilde{A}_3+\tilde{A}_3+\tilde{A}_1 +\tilde{A}_1,\ (4)\ 2\tilde{A}_2+2\tilde{A}_2+\tilde{A}_2 + \tilde{A}_2.$$ 
Every maximal parabolic diagram defines a special genus one fibration. 
\end{lemma}
\begin{proof} The assertion follows from a direct calculation.
We give an example of each type.  The last vertex is a bi-section of the genus one fibration defined by
this parabolic subdiagram.

\noindent
$(1) \  \la v_{11}, v_{41}, v_{42}, v_{32}, v_{33}, v_{23}, v_{24}, v_{14}\ra, \ 
 \la v_{(12)(34)}, v_{(13)}\ra,$\ $v_{(132)}$;

\noindent
$(2) \  \la v_{11}, v_{31}, v_{32}, v_{22}, v_{23}, v_{13}\ra,\ \la v_{(12)(34)}, v_{(124)}, v_{(142)}\ra,\ \la v_{44}, v_{(12)}\ra,$ \ $v_{14}$;

\noindent
$(3) \  \la v_{11}, v_{21}, v_{22}, v_{12}\ra,\ \la v_{33}, v_{43}, v_{44}, v_{34}\ra,\ 
\la v_{(13)(24)}, v_{(14)(23)}\ra, \la v_{(1324)}, v_{(1423)}\ra,$ \ $v_{14}$;

\noindent
$(4) \  \la v_{11}, v_{21}, v_{31}\ra,\ \la v_{42}, v_{43}, v_{44}\ra,\ 
\la v_{(14)(23)}, v_{(124)}, v_{(134)}\ra,\ \la v_{(14)}, v_{(1324)}, v_{(1234)}\ra,$ \ $v_{14}$.
\end{proof}

Since the 40 $(-2)$-classes are effective, by Lemma \ref{parabolicMII}, 
Theorem \ref{Vinberg} and Proposition \ref{FiniteIndex}, 
${\rm Aut}(S)$ is finite.  Moreover it follows from Proposition \ref{Vinberg2} that $S$ has exactly forty effective irreducible roots.  The symmetry group of $\Gamma_{\rm MII}$ is isomorphic to
$(S_4\times S_4)\cdot {\bf Z}/2{\bf Z}$.
This group is realized by automorphisms
$({\rm PGL}_2({\bf F}_3)\times {\rm PGL}_2({\bf F}_3))\cdot \iota$ of $Q={\bf P}^1\times {\bf P}^1$
where $\iota$ is the switch of the first and the second factor of $Q$.  These automorphisms can be lifted automorphisms of $S$. Obviously ${\rm Aut}(S)$ acts on $\Gamma_{\rm MII}$ 
effectively and
hence  ${\rm Aut}(S)\cong (S_4\times S_4)\cdot {\bf Z}/2{\bf Z}$.

For the $R$-invariant $(K,H)$ of $S$, we obtain $D_8$ from 8 boundary components and 
$A_2^{\oplus 2}$ from $\Gamma_1, \Gamma_2$, 
and hence $K$ contains $D_8\oplus A_2^{\oplus 2}$.  Since ${\rm det}\ K$
is divisible by $3^2$  (Proposition \ref{SSK3Pic}, Lemma \ref{overlattice}), 
$K$ is isomorphic to $E_8\oplus A_2^{\oplus 2}$ or $D_8\oplus A_2^{\oplus 2}$.
In the first case $q_{K/2K}$ is non-singular and hence we have a contradiction.
Hence $K=D_8\oplus A_2^{\oplus 2}$.  
We can easily see that the nullity of $K/2K$ is 2 and hence 
$H\cong ({\bf Z}/2{\bf Z})^2$.  

\begin{proposition}\label{MII}{\rm (Dolgachev and Kond\=o \cite[Theorem 10.6.9]{DK})}
The surface $S$ is a Coble surface of type {\rm MII} in characteristic $3$ with eight boundary components whose double cover is the supersingular $K3$ surface with Artin invariant $1$.  
The surface $S$ has forty effective irreducible roots. The automorphism group 
${\rm Aut}(S)$ is isomorphic to $(S_4\times S_4)\cdot {\bf Z}/2{\bf Z}$ and 
the $R$-invariant $(K,H)$ is $(D_8\oplus A_2^{\oplus 2}, ({\bf Z}/2{\bf Z})^2)$.
\end{proposition} 

\end{example}

\section{Possible dual graphs}\label{sec5}

In this section we prove the following theorem.

\begin{theorem}\label{dualgraph}
Assume ${\rm char}(k)\ne 2$.
Let $S$ be an Enriques or a Coble surface with finite automorphism group.
Then the dual graph $\Gamma$ of all effective irreducible roots on $S$ is of type
${\rm I}$, ${\rm II},\ldots , {\rm VII}$, ${\rm MI}$ or ${\rm MII}$.
\end{theorem}

As a corollary we obtain the following known result by the author \cite[Theorem 4.1]{Ko}
over the complex numbers and by Martin \cite[Theorem 10.1]{Martin} in any characteristic
$\ne 2$.

\begin{corollary}
Assume ${\rm char}(k)\ne 2$.  Let $Y$ be an Enriques surface with finite automorphism group.
Then the dual graph $\Gamma$ of all $(-2)$-curves on $S$ is of type
${\rm I}$, ${\rm II},\ldots , {\rm VII}$.
\end{corollary}
\begin{proof}
Martin \cite[Lemma 10.12]{Martin} proved that 
there are no special elliptic fibrations on an Enriques surface with singular fibers of type ${\rm I}_6$, $2{\rm I}_3$, $2{\rm I}_2$ and the one with four singular fibers of type $2{\rm I}_3, 2{\rm I}_3, {\rm I}_3, {\rm I}_3$.  The assertion follows from this fact and Lemmas \ref{parabolicMI}, \ref{parabolicMII}.
\end{proof}

We use the following Lemmas frequently.

\begin{lemma}\label{MultipleFiber} Let $S$ be a Coble surface and let $p : S\to {\bf P}^1$ be a genus one fibration.  Let $F$ be a reducible multiple fiber of $p$.  Then

{\rm (1)}\ 
$F$ is of type $\tilde{A}_n$ $(n\geq 1)$.

{\rm (2)}\ Assume that $F$ consists of $(-2)$-curves.  Then $F$ is of type ${\rm I}_n$.
\end{lemma}
\begin{proof}
Let $\pi : X\to S$ be the double covering branched along $B_1+\cdots +B_n$.
If $F$ contains no boundary components, then $\pi^{-1}(F)$ is connected and 
an \'etale double covering of $F$, and hence it is of type ${\rm I}_n$.  Thus we have the second assertion. In case that it contains a boundary component, the first assertion follows from 
Lemma \ref{fibration}.  
\end{proof}

\begin{lemma}\label{D4diagram}
The Dynkin diagram of type $D_4$ can be realized only as the dual graph of $(-2)$-curves.
\end{lemma}
\begin{proof}
Let $C$ be a $(-2)$-curve and let $E$ be a $(-1)$-curve meeting with two boundaries $B, B'$.  Then 
$C\cdot (2E + {1\over 2}B + {1\over 2}B')$ is an 
even number (Lemma \ref{roots2}, (2)) and hence
a diagram of type $D_4$ is generated by only $(-2)$-curves or by only $(-1)$-roots.
If $D_4$ is generated by only $(-1)$-roots, then the central $(-1)$-root of $D_4$ should contain three boundary components by Lemma \ref{roots2}, (1) which is a contradiction.
\end{proof}

\begin{lemma}\label{NonMultipleFiber}
Let $S$ be a Coble surface and let $p : S\to {\bf P}^1$ be a genus one fibration.  
Assume that $F$ is a fiber of type ${\rm II}^*, {\rm III}^*, {\rm IV}^*,
{\rm I}_n^*$.  Then a special bi-section meets only the simple components of $F$.
\end{lemma}
\begin{proof}
Since $F$ consists of $(-2)$-curves by Lemma \ref{D4diagram}, if a bi-section meets a non-simple component of $F$, then
it is a $(-2)$-curve by the proof of Lemma \ref{D4diagram}.  Then the assertion follows from Kond\=o \cite[Lemma 4.3]{Ko}.
\end{proof}

In Kond\=o \cite[\S 4]{Ko}, by using Lemmas \ref{MultipleFiber}, \ref{NonMultipleFiber}, the author determined the dual graphs of $(-2)$-curves on Enriques surfaces with finite automorphism group as follows.  First we take a special elliptic fibration $\pi:S\to {\bf P}^1$ (Lemma \ref{specialfibration}) and its special bi-section.   For each type of singular fibers of $\pi$,
we determine the dual graph or show the non-existence of such fibration inductively.  We consider the following order of types of singular fibers of a special elliptic fibration (the numbers are the ones given in \cite{Ko}):

\smallskip
\noindent
$(4.4)\ \tilde{E}_8; \ (4.5)\ \tilde{A}_8;\ (4.6)\ \tilde{D}_8;\ (4.7)\ \tilde{D}_4 + \tilde{D}_4;\ (4.8)\ \tilde{E}_7+\tilde{A}_1;\ (4.9)\ \tilde{E}_6+\tilde{A}_2;$
%
$(4.10)\ \tilde{D}_6+\tilde{A}_1+\tilde{A}_1;\ (4.11)\ \tilde{D}_5+\tilde{A}_3;$

\smallskip
\noindent
In the following we may assume that all singular fibers are of type $\tilde{A}_n$.

\smallskip
\noindent
$(4.12)\ 2\tilde{A}_7+2\tilde{A}_1;\  2\tilde{A}_7+\tilde{A}_1;\ (4.13)\ 2\tilde{A}_4+ 2\tilde{A}_4;$ 
$(4.14)\ 2\tilde{A}_5 + 2\tilde{A}_2 +\tilde{A}_1;$
%
$(4.15)\ 2\tilde{A}_5 + \tilde{A}_2 +2\tilde{A}_1;\  2\tilde{A}_3+2\tilde{A}_3 + \tilde{A}_1+\tilde{A}_1;\ (4.16)\ 2\tilde{A}_4+\tilde{A}_4;$ 
%
$(4.17) \ 2\tilde{A}_5+\tilde{A}_2 + \tilde{A}_1;$
$(4.18)\ 
2\tilde{A}_2 + 2\tilde{A}_2+\tilde{A}_2+\tilde{A}_2;\ (4.19)\ \tilde{A}_7+2\tilde{A}_1;$
$(4.20)\ 2\tilde{A}_3+\tilde{A}_3 + 2\tilde{A}_1+\tilde{A}_1;$
%
$(4.21)\ 2\tilde{A}_3+\tilde{A}_3 + \tilde{A}_1+\tilde{A}_1;$
$(4.22)\ 2\tilde{A}_2 + \tilde{A}_2+\tilde{A}_2+\tilde{A}_2;$
$(4.23)\ \tilde{A}_5 + 2\tilde{A}_2 +2\tilde{A}_1;$
%
$(4.24)\ \tilde{A}_5 + 2\tilde{A}_2 +\tilde{A}_1;\ (4.25)\ \tilde{A}_5 + \tilde{A}_2 +2\tilde{A}_1;$
$(4.26)\ \tilde{A}_3+\tilde{A}_3 + 2\tilde{A}_1+ 2\tilde{A}_1;$
%
$(4.27)\ \tilde{A}_3+\tilde{A}_3 + 2\tilde{A}_1+\tilde{A}_1;$\ 
$(4.28)$\ No multiple reducible fibers.

\smallskip
\noindent
In the steps (4.18), (4.23.2), (4.24.3), (4.24.4), (4.26), the author used some properties
of automorphisms of the covering $K3$
surface over the complex numbers to exclude the case or to reduce the case to a previous case.
On the other hand, as mentioned above, Martin \cite[Lemma 10.12]{Martin} proved that 
there are no special elliptic fibrations on an Enriques surface with singular fibers of type ${\rm I}_6$, $2{\rm I}_3$, 
$2{\rm I}_2$ (the case (4.23)) and the one with four singular fibers of type $2{\rm I}_3, 2{\rm I}_3, {\rm I}_3, {\rm I}_3$ (the case (4.18)).

Now in the following we proceed with the proof given in \cite{Ko} without using any property of the covering $K3$ surface, and show that the dual graph of type MI (resp. of type MII) appears in the case (4.18) (resp. (4.23)). We also employ the idea of Martin \cite{Martin} to use Proposition \ref{Jacobian2}.  For the other steps except (4.18), (4.23.2), (4.24.3), (4.24.4), (4.26), we refer the reader to the proof in Kondo \cite{Ko}.

\medskip
{\bf The case (4.18):} We assume that there exists a special genus one fibration with
singular fibers of type $2\tilde{A}_2+2\tilde{A}_2+\tilde{A}_2+\tilde{A}_2$ 
and a special bi-section $N$.
We will show that in this case we obtain the dual graph of type MII.
In this case we may assume that the bi-section $N$ meets the reducible fibers as in the following
Figure \ref{4.18} (in the other cases, there exists a multiple fiber of type $\tilde{E}_6$ which contradicts Lemma \ref{MultipleFiber}): 

\begin{figure}[htbp]
\begin{center}
\xy
(-30,0)*{};
@={(-10,10),(-20,10),(-20,20),(80,20),(70,10),(80,10),(20,20),(40,20),(10,10),(20,10),(40,10),(50,10),(30,0)}@@{*{\bullet}};
(-10,10)*{};(-20,10)*{}**\dir{-};
(-10,10)*{};(-20,20)*{}**\dir{-};
(-20,20)*{};(-20,10)*{}**\dir{-};
(10,10)*{};(20,10)*{}**\dir{-};
(20,20)*{};(20,10)*{}**\dir{-};
(10,10)*{};(20,20)*{}**\dir{-};
(50,10)*{};(40,10)*{}**\dir{-};
(40,20)*{};(40,10)*{}**\dir{-};
(50,10)*{};(40,20)*{}**\dir{-};
(70,10)*{};(80,10)*{}**\dir{-};
(80,20)*{};(80,10)*{}**\dir{-};
(70,10)*{};(80,20)*{}**\dir{-};
(30,0)*{};(20,10)*{}**\dir{-};
(80,10)*{};(30,0)*{}**\dir{=};
(40,10)*{};(30,0)*{}**\dir{=};
(-20,10)*{};(30,0)*{}**\dir{-};
(30,-3)*{N};(83,10)*{E};(36,10)*{E'};
\endxy
 \caption{Fiber of type $2\tilde{A}_2+2\tilde{A}_2+\tilde{A}_2+\tilde{A}_2$ and a special bi-section $N$}
 \label{4.18}
\end{center}
\end{figure}

Consider the special elliptic fibration defined by $|2(N+E)|$. 
Since $A_2+A_2+A_2$ are contained in fibers of this fibration, we may assume that 
it is of type $\tilde{A}_5+\tilde{A}_2+\tilde{A}_1$ (the case of type $\tilde{E}_7+\tilde{A}_1$ is excluded by induction because it is reduced to the previous case (4.8)).  
By considering a 2-section, the fiber of
type $\tilde{A}_5$ is not multiple.  In case that $\tilde{A}_2$ is multiple, 
a subdiagram of type $2\tilde{E}_6$ appears which contradicts Lemma \ref{MultipleFiber}.  
Thus it is of type $\tilde{A}_5+\tilde{A}_2+2\tilde{A}_1$.  We obtain
three new roots as components of this fibration.
The four bi-sections and one 4-section of this fibration meet the new roots as in the following Figure \ref{4.18-2}, because, otherwise, a multiple fiber of $2\tilde{E}_6$ appears.
In the Figure \ref{4.18-2}, we use the notation of vertices of the dual graph of Example \ref{TypeMII} because it is helpful to understand the dual graph.

\begin{figure}[htbp]
\begin{center}
\xy
(-10,0)*{};
@={(-10,10),(-20,10),(-20,20),(0,0),(20,30),(100,20),(90,10),(100,10),(20,18),(50,20),(10,10),(20,10),(40,10),(70,10),(50,30),(30,0)}@@{*{\bullet}};
(-10,10)*{};(-20,10)*{}**\dir{-};
(-10,10)*{};(-20,20)*{}**\dir{-};
(-20,20)*{};(-20,10)*{}**\dir{-};
(10,10)*{};(20,10)*{}**\dir{-};
(20,18)*{};(20,10)*{}**\dir{-};
(10,10)*{};(20,18)*{}**\dir{-};
(70,10)*{};(40,10)*{}**\dir{-};
(50,20)*{};(40,10)*{}**\dir{-};
(70,10)*{};(50,20)*{}**\dir{-};
(90,10)*{};(100,10)*{}**\dir{-};
(100,20)*{};(100,10)*{}**\dir{-};
(90,10)*{};(100,20)*{}**\dir{-};
(30,0)*{};(20,10)*{}**\dir{-};
(100,10)*{};(30,0)*{}**\dir{=};
(40,10)*{};(30,0)*{}**\dir{=};
(-10,10)*{};(0,0)*{}**\dir{-};
(0,0)*{};(10,10)*{}**\dir{-};
(-20,20)*{};(20,30)*{}**\dir{-};
(20,18)*{};(20,30)*{}**\dir{-};
(50,20)*{};(50,30)*{}**\dir{-};
(70,10)*{};(50,30)*{}**\dir{-};
(40,10)*{};(50,30)*{}**\dir{=};
(50,30)*{};(90,10)*{}**\dir{=};
(50,30)*{};(100,20)*{}**\dir{=};
(-20,10)*{};(30,0)*{}**\dir{-};
(50,30)*{};(20,10)*{}**\dir{=};
(50,30)*{};(-20,10)*{}**\dir{=};
(0,0)*{};(90,10)*{}**\dir{=};
(20,30)*{};(100,20)*{}**\dir{=};
(20,30)*{};(40,10)*{}**\dir{=};
(0,0)*{};(40,10)*{}**\dir{=};
(-6,10)*{v_{31}};(6,10)*{v_{42}};(-24,10)*{v_{11}};(25,10)*{v_{44}};(-24,20)*{v_{21}};(24,18)*{v_{43}};
(30,-3)*{v_{14}};(0,-3)*{v_{32}};(20,33)*{v_{23}};(50,33)*{v_{id}};(106,10)*{v_{(14)}};(107,20)*{v_{(1234)}};(92,5)*{v_{(1324)}};(47,8)*{v_{(14)(23)}};(75,12)*{v_{(124)}};(52,15)*{v_{(134)}};
\endxy
 \caption{}
 \label{4.18-2}
\end{center}
\end{figure}

\noindent
In Figure \ref{4.18-2} the original fibration is given by 
$$|2(v_{11}+v_{21}+v_{31})|=|2(v_{42}+v_{43}+v_{44})|=
|v_{(14)(23)}+v_{(124)}+v_{(134)}|=|v_{(14)}+v_{(1324)}+v_{(1234)}|$$ 
and $N=v_{14}$, and 
the fibration $|2(N+E)|$ is given by
$$|2(v_{14}+v_{(14)})|=|\underline{v_{id}}+v_{(124)}+v_{(134)}|=|v_{21}+\underline{v_{23}}+v_{43}+v_{42}+\underline{v_{32}}+v_{31}|,$$
where the vertices with underline are new roots.

Next we consider the special fibration defined by $|2(v_{14}+v_{(14)(23)})|$.  Then we similarly 
obtain three new roots
$$v_{(23)}, v_{22}, v_{33}$$
as components of fibers of
$$|2(v_{14}+v_{(14)(23)})|=|\underline{v_{(23)}}+v_{(1324)}+v_{(1234)}|=|v_{21}+\underline{v_{22}}+v_{42}+v_{43}+\underline{v_{33}}+v_{31}|.$$  The incidence relation between the 
new roots and the old ones follows from Lemma \ref{MultipleFiber} and the Hodge index theorem. 
Similarly we have new roots by considering the following fibrations:
$$|2(v_{32}+v_{(14)(23)})|=|v_{(14)}+v_{(1234)}+\underline{v_{(1423)}}|=|v_{11}+\underline{v_{13}}+v_{43}+v_{44}+\underline{v_{24}}+v_{21}|,$$
$$|2(v_{23}+v_{(14)(23)})|=|v_{(14)}+v_{(1324)}+\underline{v_{(1432)}}|=|v_{11}+\underline{v_{12}}+v_{42}+v_{44}+\underline{v_{34}}+v_{31}|,$$
$$|2(v_{14}+v_{(1423)})|=|v_{12}+v_{13}+v_{43}+\underline{v_{41}}+v_{21}+v_{22}|.$$
We have now obtained all vertices of $\Gamma_0$.

We continue this process.  
The following are fibrations used and the obtained roots which are marked by underlines.
We remark that in each step we obtain exactly one new root.  Since the fiber containing the new root is non-multiple and hence the intersection numbers of the obtained root with the known roots are determined uniquely.

\smallskip
(1) \ $|2(v_{13}+v_{(134)})|=|v_{(14)}+v_{(1234)}+ \underline{v_{(34)}}|$;

\smallskip
(2) \ $|2(v_{13}+v_{23}+v_{43})|=|2(v_{31}+v_{32}+v_{34})|=|v_{id}+v_{(124)}+\underline{v_{(142)}}|$; 

\smallskip
(3) \ $|2(v_{24}+v_{34}+v_{44})|=|2(v_{11}+v_{12}+v_{13})|=|v_{(142)}+v_{(14)(23)}+\underline{v_{(143)}}|$; 

\smallskip
(4) \ $|2(v_{34}+v_{(34)})|=|v_{(143)}+v_{id}+\underline{v_{(243)}}|$; 

\smallskip
(5) \ $|2(v_{31}+v_{32}+v_{33})|=|2(v_{14}+v_{24}+v_{44})|=|v_{(34)}+v_{(1234)}+\underline{v_{(1342)}}|$;

\smallskip
(6) \ $|2(v_{21}+v_{31}+v_{41})|=|2(v_{12}+v_{13}+v_{14})|=|v_{(23)}+v_{(34)}+\underline{v_{(24)}}|$; 

\smallskip
(7) \ $|2(v_{12}+v_{32}+v_{42})|=|2(v_{21}+v_{23}+v_{24})|=|v_{(14)}+v_{(34)}+\underline{v_{(13)}}|$;

\smallskip
(8) \ $|2(v_{14}+v_{24}+v_{34})|=|2(v_{41}+v_{42}+v_{43})|=|v_{(23)}+v_{(13)}+\underline{v_{(12)}}|$; 

\smallskip
(9) \ $|2(v_{13}+v_{23}+v_{33})|=|2(v_{41}+v_{42}+v_{44})|=|v_{(34)}+v_{(1432)}+\underline{v_{(1243)}}|=|v_{(143)}+v_{(243)}+ \underline{v_{(12)(34)}}|$;

\smallskip
(10) \ $|2(v_{14}+v_{24}+v_{44})|=|2(v_{31}+v_{32}+v_{33})|=|v_{(134)}+v_{(12)(34)}+\underline{v_{(234)}}|$; 

\smallskip
(11) \ $|2(v_{12}+v_{22}+v_{42})|=|2(v_{31}+v_{33}+v_{34})|=|v_{(14)(23)}+v_{(243)}+\underline{v_{(132)}}|$; 

\smallskip
(12) \ $|2(v_{13}+v_{33}+v_{43})|=|2(v_{21}+v_{22}+v_{24})|=|v_{(14)(23)}+v_{(234)}+\underline{v_{(123)}}|$; 

\smallskip
(13) \ $|2(v_{14}+v_{34}+v_{44})|=|2(v_{21}+v_{22}+v_{23})|=|v_{(243)}+v_{(124)}+\underline{v_{(13)(24)}}|$; 

\smallskip
\noindent
Thus we have obtained 40 roots of $\Gamma_{\rm MII}$.

\medskip
{\bf The case (4.23.2):}\ We assume that there exists a special elliptic fibration with singular
fibers of type $\tilde{A}_5+2\tilde{A}_2+2\tilde{A}_1$ and a special bi-section $N$ as in the following Figure \ref{4.23}:

\begin{figure}[htbp]
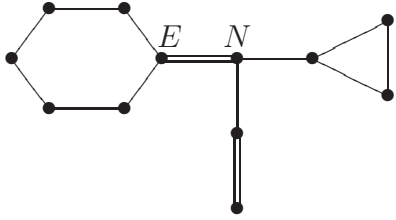

\begin{center}
\xy
(-30,0)*{};
@={(0,0),(5,-6.65),(5,6.65),(15,6.65),(15,-6.65),(20,0),(30,0),(30,-10),(30,-20),(40,0),(50,-5),(50,5)}@@{*{\bullet}};
(0,0)*{};(5,-6.65)*{}**\dir{-};
(0,0)*{};(5,6.65)*{}**\dir{-};
(5,-6.65)*{};(15,-6.65)*{}**\dir{-};
(5,6.65)*{};(15,6.65)*{}**\dir{-};
(15,6.65)*{};(20,0)*{}**\dir{-};
(15,-6.65)*{};(20,0)*{}**\dir{-};
(20,0)*{};(30,0)*{}**\dir{=};
(30,0)*{};(30,-10)*{}**\dir{-};
(30,-10)*{};(30,-20)*{}**\dir{=};
(30,0)*{};(40,0)*{}**\dir{-};
(40,0)*{};(50,5)*{}**\dir{-};
(40,0)*{};(50,-5)*{}**\dir{-};
(50,-5)*{};(50,5)*{}**\dir{-};
(30,3)*{N};(21,3)*{E};
\endxy
 \caption{Fiber of type $\tilde{A}_5+2\tilde{A}_2+2\tilde{A}_1$ and a special bi-section $N$}
 \label{4.23}
\end{center}
\end{figure}

We will show that in this case we obtain the dual graph of type MI.
Consider the special elliptic fibration defined by $|2(N+E)|$. 
Since $A_3+A_2+A_1$ are contained in the fibers of this fibration, we may assume that it is of type 
$\tilde{A}_3+\tilde{A}_3+\tilde{A}_1+\tilde{A}_1$ or $\tilde{A}_5+\tilde{A}_2+\tilde{A}_1$.
In case $\tilde{A}_5+\tilde{A}_2+\tilde{A}_1$, we can reduce this case to the previous case
(see Kond\=o \cite[4.23.2]{Ko}).  Thus we may assume that it is of type $\tilde{A}_3+\tilde{A}_3+ 2\tilde{A}_1+\tilde{A}_1$ (in case that a fiber of type $\tilde{A}_3$ is multiple we reduce this case to the previous case (4.20)), and obtain four new roots as components of the fibers.
The four bi-sections meet the new roots as in the following Figure \ref{4.23-2} by Lemma \ref{MultipleFiber}.  In the Figure \ref{4.23-2}, we use the notation of vertices of the dual graph of Example \ref{TypeMI} which is helpful to understand the dual graph:

\begin{figure}[htbp]
\begin{center}
\xy
(-20,0)*{};
@={(-15,0),(2.5,10),(-6,-16),(-6,16),(10,16),(10,-16),(22,0),(40,0),(40,-20),(40,-30),(60,0),(70,-5),(70,5),(80,10),(80,-10),(40,-40)}@@{*{\bullet}};
(-15,0)*{};(-6,-16)*{}**\dir{-};
(-15,0)*{};(-6,16)*{}**\dir{-};
(-6,-16)*{};(10,-16)*{}**\dir{-};
(-6,16)*{};(10,16)*{}**\dir{-};
(10,16)*{};(22,0)*{}**\dir{-};
(10,-16)*{};(22,0)*{}**\dir{-};
(22,0)*{};(40,0)*{}**\dir{=};
(40,0)*{};(40,-20)*{}**\dir{-};
(40,-20)*{};(40,-30)*{}**\dir{=};
(40,0)*{};(60,0)*{}**\dir{-};
(60,0)*{};(70,5)*{}**\dir{-};
(60,0)*{};(70,-5)*{}**\dir{-};
(70,-5)*{};(70,5)*{}**\dir{-};
(80,10)*{};(80,-10)*{}**\dir{-};
(70,5)*{};(80,10)*{}**\dir{-};
(70,-5)*{};(80,-10)*{}**\dir{-};
(40,-30)*{};(40,-40)*{}**\dir{=};
(-6,16)*{};(2.5,10)*{}**\dir{-};
(-6,-16)*{};(2.5,10)*{}**\dir{-};
(10,16)*{};(2.5,10)*{}**\dir{-};
(10,-16)*{};(2.5,10)*{}**\dir{-};
(40,-40)*{};(60,0)*{}**\dir{=};
(10,16)*{};(80,10)*{}**\dir{=};
(10,-16)*{};(80,-10)*{}**\dir{=};
(10,-16)*{};(40,-40)*{}**\dir{=};
(10,16)*{};(40,-40)*{}**\dir{=};
(2.5,10)*{};(60,0)*{}**\dir{=};
(40,-20)*{};(2.5,10)*{}**\dir{=};
(40,-20)*{};(80,10)*{}**\dir{-};
(40,-20)*{};(80,-10)*{}**\dir{-};
(-19,0)*{v_{45}};(-2,10)*{v_{36}};(25,-3)*{v_{12}};(10,19)*{v_{16}};(10,-19)*{v_{23}};(-6,-19)*{v_{34}};
(-6,19)*{v_{56}};(46,-3)*{v_{12,35}};(59,3)*{v_{15,24}};(65,-8)*{v_{14,26}};(65,8)*{v_{13,25}};(86,10)*{v_{16,24}};(86,-10)*{v_{15,23}};(34,-20)*{v_{14,25}};(36,-26)*{v_{135}};(40,-43)*{v_{146}};
\endxy
 \caption{}
 \label{4.23-2}
\end{center}
\end{figure}

\noindent
In Figure \ref{4.23-2} the original fibration is given by 
$$|v_{12}+v_{23}+v_{34}+v_{45}+v_{56}+v_{16}|=|2(v_{14,25}+v_{135})|=
|2(v_{13,25}+v_{14,26}+v_{15,24})|$$ 
and $N=v_{(12,35)}$, and 
the fibration $|2(N+E)|$ is given by
$$|2(v_{12}+v_{12,35})|=|v_{34}+v_{45}+v_{56}+\underline{v_{36}}|=|v_{13,25}+v_{14,26}+\underline{v_{15,23}}+\underline{v_{16,24}}| = 
|v_{135}+\underline{v_{146}}|,$$
where the vertices with underline are new.

Next we consider the special genus one fibration defined by $|2(v_{16}+v_{16,24})|$.  
Then we obtain three new roots
$$v_{25}, v_{13,24}, v_{124}$$
as components of fibers of 
$$|2(v_{16}+v_{16,24})|=|v_{23}+v_{34}+v_{45}+\underline{v_{25}}|=
|v_{12,35}+v_{15,24}+v_{14,26}+\underline{v_{13,24}}|=|v_{135}+\underline{v_{124}}|.$$

We continue this process.  The following are fibrations used and the obtained new roots which are marked by underline.  We remark that in each step except the case (13) we obtain exactly one new root in a fiber and hence the intersection numbers of the obtained root with the known roots are determined uniquely.

\smallskip
(1) \ $|2(v_{23}+v_{15,23})|=|v_{16}+v_{56}+v_{45}+\underline{v_{14}}|=|v_{12,35}+v_{15,24}+v_{13,25}+ \underline{v_{15,26}}|=|v_{135}+\underline{v_{125}}|$; 

\smallskip
(2) \ $|2(v_{34}+v_{15,26})|=|v_{12}+v_{25}+v_{56}+v_{16}|=|v_{15,24}+v_{14,26}+v_{15,23}+\underline{v_{13,26}}|=|v_{135}+\underline{v_{145}}|$;

\smallskip
(3) \ $|2(v_{23}+v_{146})|=|v_{12,35}+v_{14,25}+v_{16,24}+v_{13,25}+v_{14,26}+ \underline{v_{16,25}}|$; 

\smallskip
(4) \ $|2(v_{45}+v_{13,26})|=|v_{12}+v_{16}+v_{36}+v_{23}|=|v_{13,24}+v_{14,26}+v_{13,25}+v_{15,26}|=
|v_{135}+\underline{v_{134}}|$;

\smallskip
(5) \ $|2(v_{56}+v_{13,24})|=|v_{12}+v_{14}+v_{34}+v_{23}|=|v_{13,25}+v_{15,24}+v_{13,26}+v_{16,24}|=
|v_{135}+\underline{v_{136}}|$;

\smallskip
(6) \ $|2(v_{16}+v_{16,25})|=|v_{125}+v_{134}|=|v_{14,25}+v_{15,26}+v_{13,25}+\underline{v_{12,34}}|$;

\smallskip
(7) \ $|2(v_{25}+v_{16,25})|=|v_{12,34}+v_{16,24}+v_{15,26}+\underline{v_{16,23}}|$;

\smallskip
(8) \ $|2(v_{12}+v_{12,34})|=|v_{15,26}+v_{13,24}+v_{16,25}+\underline{v_{14,23}}|$;

\smallskip
(9) \ $|2(v_{25}+v_{14,25})|=|v_{14,23}+v_{15,24}+v_{14,26}+\underline{v_{12,36}}|$;

\smallskip
\noindent
Note that we have now obtained all fifteen roots indexed by synthemes.

\smallskip
(10) \ $|2(v_{14}+v_{14,26})|=|v_{23}+v_{36}+v_{56}+v_{25}|=|v_{12,35}+v_{13,26}+v_{16,24}+v_{15,26}|=
|v_{135}+\underline{v_{126}}|$;

\smallskip
(11) \ $|2(v_{25}+v_{13,25})|=|v_{14}+v_{16}+v_{36}+v_{34}|=|v_{12,35}+v_{13,24}+v_{13,26}+v_{15,23}|=
|v_{135}+\underline{v_{123}}|$;

\smallskip
(12) \ $|2(v_{36}+v_{15,24})|=|v_{12}+v_{14}+v_{45}+v_{25}|=|v_{13,24}+v_{15,23}+v_{16,24}+v_{15,26}|=
|v_{135}+\underline{v_{156}}|$;

\smallskip
\noindent
Note that we have now obtained all ten roots indexed by triads. 

\smallskip
(13) \ $|2(v_{12}+v_{12,34})|=|v_{134}+v_{156}|=|v_{13,24}+v_{15,26}+v_{14,23}+v_{16,25}|=
|v_{36}+v_{45}+\underline{v_{35}}+\underline{v_{46}}|$;

\smallskip
\noindent
This case is the same as in the first case $|2(N+E)|$.   
We may assume that the graph is the desired one.

\smallskip
(14) \ $|2(v_{16}+v_{16,25})|=|v_{125}+v_{134}|=|v_{12,34}+v_{13,25}+v_{15,34}+v_{14,25}|=|v_{23}+v_{35}+v_{45}+\underline{v_{24}}|$;

\smallskip 
(15) \ $|2(v_{25}+v_{14,25})|=|v_{124}+v_{145}|=|v_{12,36}+v_{14,26}+v_{15,24}+v_{14,23}|=|v_{16}+v_{46}+v_{34}+\underline{v_{13}}|$;

\smallskip 
(16) \ $|2(v_{46}+v_{13,25})|=|v_{134}+v_{136}|=|v_{13,24}+v_{14,25}+v_{13,26}+v_{16,25}|=|v_{12}+v_{23}+v_{35}+\underline{v_{15}}|$;

\smallskip
(17) \ $|2(v_{45}+v_{16,23})|=|v_{146}+v_{156}|=|v_{14,23}+v_{16,24}+v_{15,23}+v_{16,25}|=
|v_{12}+v_{13}+v_{36}+\underline{v_{26}}|$;

\smallskip
\noindent
Thus we have obtained 40 roots of $\Gamma_{\rm MI}$.

\bigskip
{\bf The case (4.24.3):}\ We start with the following Figure \ref{4.24}:

\begin{figure}[htbp]
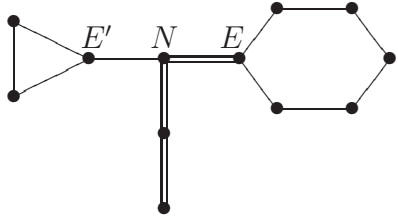

\begin{center}
\xy
(-30,0)*{};
@={(60,0),(55,-6.65),(55,6.65),(45,6.65),(45,-6.65),(20,0),(30,0),(30,-10),(30,-20),(40,0),(10,-5),(10,5)}@@{*{\bullet}};
(60,0)*{};(55,-6.65)*{}**\dir{-};
(60,0)*{};(55,6.65)*{}**\dir{-};
(55,-6.65)*{};(45,-6.65)*{}**\dir{-};
(55,6.65)*{};(45,6.65)*{}**\dir{-};
(45,6.65)*{};(40,0)*{}**\dir{-};
(45,-6.65)*{};(40,0)*{}**\dir{-};
(20,0)*{};(30,0)*{}**\dir{-};
(30,0)*{};(30,-10)*{}**\dir{=};
(30,-10)*{};(30,-20)*{}**\dir{=};
(30,0)*{};(40,0)*{}**\dir{=};
(20,0)*{};(10,5)*{}**\dir{-};
(20,0)*{};(10,-5)*{}**\dir{-};
(10,-5)*{};(10,5)*{}**\dir{-};
(30,3)*{N};(21,3)*{E'};(39,3)*{E};
\endxy
 \caption{Fiber of type $\tilde{A}_5+2\tilde{A}_2+\tilde{A}_1$ and a special bi-section $N$}
 \label{4.24}
\end{center}
\end{figure}

Consider the fibration defined by $|2(E+N)|$ whose fibers contain $A_3+A_2+A_1$.  
Therefore the fibration has singular fibers of type $\tilde{A}_5+\tilde{A}_2 + \tilde{A}_1$
or $\tilde{A}_3+\tilde{A}_3 + \tilde{A}_1 +\tilde{A}_1$.  The first case reduces to a  previous case by Kond\=o \cite[(4.24.3), (a)]{Ko}.  
Thus we may assume that it has singular fibers of type 
$\tilde{A}_3+\tilde{A}_3 + 2\tilde{A}_1 +\tilde{A}_1$ as in the following Figure \ref{4.24-2}:

\begin{figure}[htbp]
\begin{center}
\xy
(-30,0)*{};
@={(00,0),(10,0),(25,10),(40,0),(50,0),(0,10),(10,10),(20,0),(20,-10),(30,-10),(30,0),(40,10),(50,10)}@@{*{\bullet}};
(0,0)*{};(10,0)*{}**\dir{-};
(0,0)*{};(0,10)*{}**\dir{-};
(10,0)*{};(10,10)*{}**\dir{-};
(0,10)*{};(10,10)*{}**\dir{-};
(10,10)*{};(25,10)*{}**\dir{-};
(25,10)*{};(40,10)*{}**\dir{=};
(25,10)*{};(20,0)*{}**\dir{-};
(25,10)*{};(30,0)*{}**\dir{=};
(25,10)*{};(10,0)*{}**\dir{-};
(40,0)*{};(40,10)*{}**\dir{-};
(40,0)*{};(50,0)*{}**\dir{-};
(50,0)*{};(50,10)*{}**\dir{-};
(40,10)*{};(50,10)*{}**\dir{-};
(20,0)*{};(20,-10)*{}**\dir{=};
(30,0)*{};(30,-10)*{}**\dir{=};
(17,0)*{N};(17,-10)*{E};(25,13)*{E'};
\endxy 
 \caption{}
 \label{4.24-2}
\end{center}
\end{figure}

If the bi-section $E'$ is not a $(-2)$-curve, then $N$ and the components of the fiber 
in left hand side are not $(-2)$-curves because the intersection number of 
a $(-2)$-root and a $(-1)$-root is even, and the curves are as in the following 
Figure \ref{4.24-3} in which the dotted lines are $(-1)$-curves, $N_1,\ldots, N_6$ are 
$(-4)$-curves and the remaining ones are $(-2)$-curves.

\begin{figure}[htbp]
\begin{center}
\xy
(-30,0)*{};
(-10,0)*{};(65,0)*{}**\dir{--};
(-5,-5)*{};(-5,25)*{}**\dir{-};
(0,20)*{};(-25,20)*{}**\dir{--};
(0,5)*{};(-25,5)*{}**\dir{--};
(-20,15)*{};(-20,30)*{}**\dir{-};
(-20,10)*{};(-20,-5)*{}**\dir{-};
(-15,0)*{};(-35,0)*{}**\dir{--};
(-15,25)*{};(-35,25)*{}**\dir{--};
(-30,30)*{};(-30,-5)*{}**\dir{-};
(10,-5)*{};(10,15)*{}**\dir{-};
(5,10)*{};(25,10)*{}**\dir{--};
(5,10)*{};(25,10)*{}**\dir{--};
(15,5)*{};(15,20)*{}**\dir{-};
(20,5)*{};(20,20)*{}**\dir{-};
(35,-5)*{};(35,20)*{}**\dir{-};
(50,-5)*{};(50,20)*{}**\dir{-};
(60,5)*{};(60,20)*{}**\dir{-};
(40,15)*{};(30,15)*{}**\dir{-};
(30,15)*{};(30,5)*{}**\dir{-};
(40,5)*{};(30,5)*{}**\dir{-};
(45,10)*{};(65,10)*{}**\dir{-};
(45,15)*{};(65,15)*{}**\dir{-};
(68,0)*{C};(14,23)*{E};
(-5,-8)*{N_3};(-20,-8)*{N_2};
(-30,-8)*{N_1};(-17,15)*{N_4};
(10,-8)*{N_5};(21,23)*{N_6};(3,12)*{C'};
\endxy 
 \caption{}
 \label{4.24-3}
\end{center}
\end{figure}

\noindent
Here $N=2C'+{1\over 2}N_5+{1\over 2}N_6$ and $E'= 2C+{1\over 2}N_3+{1\over 2}N_5$.
Contracting the 5 $(-1)$-curves except $C$, and then contracting the image of $E$, we get a Jacobian elliptic fibration with singular fibers
of type ${\rm I}_4, {\rm I}_4, {\rm I}_2, {\rm III}$.  
However there are no such fibrations by
Table \ref{extremal}.  Thus this case does not occur.

In case that the bi-section $E'$ is a $(-2)$-curve, then by using Proposition \ref{Jacobian2} and by applying the same method by Martin \cite[p.640, the case (c)]{Martin} we
can find an effective irreducible root $F$ as in the following Figure \ref{4.24-4}:

\begin{figure}[htbp]
\begin{center}
\xy
(-30,0)*{};
@={(00,0),(10,0),(25,10),(40,0),(50,0),(0,10),(10,10),(20,0),(20,-10),(30,-10),(30,0),(40,10),(50,10),
(10,20)}@@{*{\bullet}};
(0,0)*{};(10,0)*{}**\dir{-};
(0,0)*{};(0,10)*{}**\dir{-};
(10,0)*{};(10,10)*{}**\dir{-};
(0,10)*{};(10,10)*{}**\dir{-};
(10,10)*{};(25,10)*{}**\dir{-};
(25,10)*{};(40,10)*{}**\dir{=};
(25,10)*{};(20,0)*{}**\dir{-};
(25,10)*{};(30,0)*{}**\dir{=};
(25,10)*{};(10,0)*{}**\dir{-};
(40,0)*{};(40,10)*{}**\dir{-};
(40,0)*{};(50,0)*{}**\dir{-};
(50,0)*{};(50,10)*{}**\dir{-};
(40,10)*{};(50,10)*{}**\dir{-};
(20,0)*{};(20,-10)*{}**\dir{=};
(30,0)*{};(30,-10)*{}**\dir{=};
(10,20)*{};(20,0)*{}**\dir{-};
(10,20)*{};(30,-10)*{}**\dir{=};
(10,20)*{};(0,0)*{}**\dir{-};
(10,20)*{};(0,10)*{}**\dir{-};
(10,20)*{};(40,10)*{}**\dir{=};
(7,20)*{F};
\endxy 
\caption{}
\label{4.24-4}
\end{center}
\end{figure}

\noindent
Then, as pointed out by Martin, we can easily find that there exists a fibration with two multiple singular fibers of type $2\tilde{A}_2$ and $2\tilde{A}_1$, and hence reduce this case to the previous case.

\medskip
{\bf The case (4.24.4):}\ We start with the following Figure \ref{4.24-5}:

\begin{figure}[htbp]
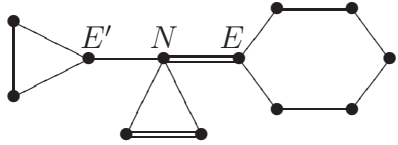

\begin{center}
\xy
(-20,0)*{};
@={(60,0),(55,-6.65),(55,6.65),(45,6.65),(45,-6.65),(20,0),(30,0),(25,-10),(35,-10),(40,0),(10,-5),(10,5)}@@{*{\bullet}};
(60,0)*{};(55,-6.65)*{}**\dir{-};
(60,0)*{};(55,6.65)*{}**\dir{-};
(55,-6.65)*{};(45,-6.65)*{}**\dir{-};
(55,6.65)*{};(45,6.65)*{}**\dir{-};
(45,6.65)*{};(40,0)*{}**\dir{-};
(45,-6.65)*{};(40,0)*{}**\dir{-};
(20,0)*{};(30,0)*{}**\dir{-};
(35,-10)*{};(25,-10)*{}**\dir{=};
(30,0)*{};(35,-10)*{}**\dir{-};
(30,0)*{};(25,-10)*{}**\dir{-};
(30,0)*{};(40,0)*{}**\dir{=};
(20,0)*{};(10,5)*{}**\dir{-};
(20,0)*{};(10,-5)*{}**\dir{-};
(10,-5)*{};(10,5)*{}**\dir{-};
(30,3)*{N};(21,3)*{E'};(39,3)*{E};
\endxy
 \caption{Fiber of type $\tilde{A}_5+2\tilde{A}_2+\tilde{A}_1$ and a special bi-section $N$}
 \label{4.24-5}
\end{center}
\end{figure}

By considering the fibration defined by $|2(E+N)|$, we have an elliptic fibration with
singular fibers of type $\tilde{A}_3+\tilde{A}_3+2\tilde{A}_1+\tilde{A}_1$.
We may assume that a bi-section $E'$ meets the components of fibers as in the following Figure \ref{4.24-1} because otherwise we obtain the Figure \ref{4.24-2} and reduce this case to the previous case (4.24), (iii):

\begin{figure}[htbp]
\begin{center}
\xy
(-30,0)*{};
@={(00,0),(10,0),(25,10),(40,0),(50,0),(0,10),(10,10),(20,0),(20,-10),(30,-10),(30,0),(40,10),(50,10)}@@{*{\bullet}};
(0,0)*{};(10,0)*{}**\dir{-};
(0,0)*{};(0,10)*{}**\dir{-};
(10,0)*{};(10,10)*{}**\dir{-};
(0,10)*{};(10,10)*{}**\dir{-};
(10,10)*{};(25,10)*{}**\dir{-};
(25,10)*{};(40,10)*{}**\dir{=};
(25,10)*{};(20,0)*{}**\dir{-};
(25,10)*{};(30,0)*{}**\dir{-};
(25,10)*{};(10,0)*{}**\dir{-};
(25,10)*{};(30,-10)*{}**\dir{-};
(40,0)*{};(40,10)*{}**\dir{-};
(40,0)*{};(50,0)*{}**\dir{-};
(50,0)*{};(50,10)*{}**\dir{-};
(40,10)*{};(50,10)*{}**\dir{-};
(20,0)*{};(20,-10)*{}**\dir{=};
(30,0)*{};(30,-10)*{}**\dir{=};
(25,13)*{E'};(17,-10)*{E};(17,0)*{N};
\endxy 
\caption{}
\label{4.24-1}
\end{center}
\end{figure}

In this case $E'$ is a $(-2)$-curve by Lemma \ref{D4diagram}.  By using the method of Martin \cite{Martin} we can see that
this case does not occur (see Martin \cite[Table 5]{Martin}).

\medskip
{\bf The case (4.26):} \ We start with the following Figure \ref{4.26}:

\begin{figure}[htbp]
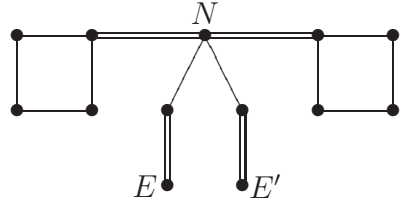

\begin{center}
\xy
(-30,0)*{};
@={(00,0),(10,0),(25,10),(40,0),(50,0),(0,10),(10,10),(20,0),(20,-10),(30,-10),(30,0),(40,10),(50,10)}@@{*{\bullet}};
(0,0)*{};(10,0)*{}**\dir{-};
(0,0)*{};(0,10)*{}**\dir{-};
(10,0)*{};(10,10)*{}**\dir{-};
(0,10)*{};(10,10)*{}**\dir{-};
(10,10)*{};(25,10)*{}**\dir{=};
(25,10)*{};(40,10)*{}**\dir{=};
(25,10)*{};(20,0)*{}**\dir{-};
(25,10)*{};(30,0)*{}**\dir{-};
(40,0)*{};(40,10)*{}**\dir{-};
(40,0)*{};(50,0)*{}**\dir{-};
(50,0)*{};(50,10)*{}**\dir{-};
(40,10)*{};(50,10)*{}**\dir{-};
(20,0)*{};(20,-10)*{}**\dir{=};
(30,0)*{};(30,-10)*{}**\dir{=};
(25,13)*{N};
(17,-10)*{E};(33,-10)*{E'};
\endxy 
\caption{Fiber of type $\tilde{A}_3+\tilde{A}_3+2\tilde{A}_1+2\tilde{A}_1$ and a special bi-section $N$}
\label{4.26}
\end{center}
\end{figure}

If the root $N$ is not a $(-2)$-curve, 
then by considering the intersection number of a $(-2)$-root and a $(-1)$-root, 
$E, E'$ are $(-2)$-curves and the configuration of curves is as in
the following Figure \ref{4.26-2}:

\begin{figure}[htbp]
\begin{center}
\xy
(-30,0)*{};
(-20,0)*{};(65,0)*{}**\dir{--};
(-10,-5)*{};(-10,20)*{}**\dir{-};
(-20,5)*{};(-20,20)*{}**\dir{-};
(-5,10)*{};(-25,10)*{}**\dir{-};
(-5,15)*{};(-25,15)*{}**\dir{-};
(5,-5)*{};(5,15)*{}**\dir{-};
(0,10)*{};(20,10)*{}**\dir{--};
(10,5)*{};(10,20)*{}**\dir{-};
(15,5)*{};(15,20)*{}**\dir{-};
(30,-5)*{};(30,15)*{}**\dir{-};
(25,10)*{};(45,10)*{}**\dir{--};
(35,5)*{};(35,20)*{}**\dir{-};
(40,5)*{};(40,20)*{}**\dir{-};
(55,-5)*{};(55,20)*{}**\dir{-};
(65,5)*{};(65,20)*{}**\dir{-};
(50,10)*{};(70,10)*{}**\dir{-};
(50,15)*{};(70,15)*{}**\dir{-};
(68,0)*{N'};(10,23)*{E};
(35,23)*{E'};
(5,18)*{N_1};(16,23)*{N_2};
(30,18)*{N_3};(41,23)*{N_4};
\endxy 
\caption{}
\label{4.26-2}
\end{center}
\end{figure}

\noindent
Here the dotted lines are $(-1)$-curves, $N_1,\ldots, N_4$ are $(-4)$-curves and the remaining ones are $(-2)$-curves.  Note that $N=2N' +{1\over 2}N_1 +{1\over 2}N_3$.

By contracting all $(-1)$-curves except $N'$ and then contracting the images of $E, E'$, we have a Jacobian fibration with singular fibers of type ${\rm I}_4, {\rm I}_4, {\rm III}, {\rm III}$.  By Table \ref{extremal}, there are no such
fibrations and hence this case does not occur.

Finally in case that $N$ is a $(-2)$-curve, by the same method of Martin \cite[p.639, the case (b)]{Martin}, we can see
that there exit three effective roots as in the following Figure \ref{4.26-3}:

\begin{figure}[htbp]
\begin{center}
\xy
(-30,0)*{};
@={(-5,0),(5,0),(25,10),(25,20),(45,0),(55,0),(-5,10),(5,10),(20,0),(20,-10),(30,-10),(30,0),(45,10),(55,10),(10,-15),(40,-15)}@@{*{\bullet}};
(-5,0)*{};(5,0)*{}**\dir{-};
(-5,0)*{};(-5,10)*{}**\dir{-};
(5,0)*{};(5,10)*{}**\dir{-};
(-5,10)*{};(5,10)*{}**\dir{-};
(5,10)*{};(25,10)*{}**\dir{=};
(25,10)*{};(45,10)*{}**\dir{=};
(25,10)*{};(20,0)*{}**\dir{-};
(25,10)*{};(30,0)*{}**\dir{-};
(45,0)*{};(45,10)*{}**\dir{-};
(45,0)*{};(55,0)*{}**\dir{-};
(55,0)*{};(55,10)*{}**\dir{-};
(45,10)*{};(55,10)*{}**\dir{-};
(20,0)*{};(20,-10)*{}**\dir{=};
(30,0)*{};(30,-10)*{}**\dir{=};
(25,20)*{};(25,10)*{}**\dir{=};
(25,20)*{};(20,0)*{}**\dir{-};
(25,20)*{};(30,0)*{}**\dir{-};
(25,20)*{};(-5,0)*{}**\dir{=};
(25,20)*{};(55,0)*{}**\dir{=};
(10,-15)*{};(40,-15)*{}**\dir{=};
(10,-15)*{};(20,0)*{}**\dir{-};
(10,-15)*{};(30,0)*{}**\dir{-};
(40,-15)*{};(20,0)*{}**\dir{-};
(40,-15)*{};(30,0)*{}**\dir{-};
(10,-15)*{};(5,10)*{}**\dir{=};
(10,-15)*{};(10,-20.5)*{}**\dir{=};
(9.5,-20)*{};(55.5,-20)*{}**\dir{=};
(55,-20.5)*{};(55,0)*{}**\dir{=};
(40,-15)*{};(40,-25.5)*{}**\dir{=};
(40.5,-25)*{};(-5.5,-25)*{}**\dir{=};
(-5,0)*{};(-5,-25.5)*{}**\dir{=};
(40,-15)*{};(45,10)*{}**\dir{=};
(40,-15)*{};(30,0)*{}**\dir{-};
(30,12)*{N};(29,21)*{F};(44,-15)*{F'};(6,-15)*{F''};
\endxy 
\caption{}
\label{4.26-3}
\end{center}
\end{figure}

This is the same situation as in Martin \cite{Martin}, that is, 
the fibration defined by $|F+N|=|F'+F''|$ contains 
eight disjoint roots as its fibers.  There are no fibrations satisfying this property.
Therefore this case does not occur.

Thus we have now finished the proof of Theorem \ref{dualgraph}.

\section{The possible characteristics and the number of boundary components}\label{sec6}

Let $S$ be a Coble surface in characteristic $p\ne 2$ with the dual graph I, II, \ldots, VII, MI or
MII.  We will determine the characteristic $p$ and the number of boundary components. 

First we start the following Lemma.

\begin{lemma}\label{CM-prop}
For a Coble surface $S$ with the dual graph of type {\rm I, II, \ldots, VII}, 
the Coble-Mukai lattice ${\rm CM}(S)$ is isomorphic to $E_{10}$.  For a Coble surface $S$ with the dual graph of type 
{\rm MI, MII}, the determinant of the Coble-Mukai lattice ${\rm CM}(S)$ is equal to $-2^l$ for some non-negative integer $l$.
\end{lemma}
\begin{proof}
Let $Y$ be an Enriques surface.  Then ${\rm Num}(Y)$ is isomorphic to $E_{10}$.  
If $Y$ is of type
I, II, \ldots or VII, then all $(-2)$-curves on $Y$ appear as vertices of the dual graph and
all genus one fibrations appear as maximal parabolic subdiagrams of the dual graph.  
The Enriques' reducibility
lemma (Cossec, Dolgachev and Liedtke \cite[Theorem 2.3.5]{CDL}) implies that $(-2)$-curves on $Y$ generate $E_{10}$.  It follows that effective 
irreducible roots on a Coble surface $S$ of type I, II, \ldots, VII generate $E_{10}$, 
and hence ${\rm CM}(S)$ is so.

In case of Coble surface $S$ of type MI or MII, consider a parabolic subdiagram 
of type $\tilde{A}_7+\tilde{A}_1$ or $\tilde{A}_3+\tilde{A}_3+\tilde{A}_1+\tilde{A}_1$, respectively,
which defines a special elliptic fibration on $S$.
Together with a bi-section it generates a sublattice of ${\rm CM}(S)$ of finite index whose determinant
is $\pm 2^k$.  This implies the assertion.
\end{proof}

(1) \ {\bf The case of type I.}\quad  
There exists a parabolic subdiagram of type $\tilde{E}_8$ in Figure \ref{typeI}.
It follows from Table \ref{extremal}, Propositions \ref{quasi-ell} and \ref{Halphen} that 
$S$ is obtained from a genus one fibration with singular fibers of type $({\rm II}^*)$, 
$({\rm II}^*, {\rm II})$, $({\rm II}^*, {\rm I}_1)$ or $({\rm II}^*, {\rm I}_1, {\rm I}_1)$ by blowing up the singular points of the fibers of type ${\rm II}$ or ${\rm I}_1$.  
This implies that the number $n$ of boundary components is 1 or 2.  
Since the diagram of type I contains a subdiagram of type $D_4$, by Lemma \ref{D4diagram} and Lemma \ref{roots2}, (2), 
the ten vertices in Figure \ref{typeI} except the two circles are represented by $(-2)$-curves.  If $n=1$, then there are no 
$(-1)$-roots and hence the remaining two vertices are represented by $(-2)$-curves, too.
In case of $n=2$, consider a parabolic subdiagram of type $\tilde{E}_7+\tilde{A}_1$ which
is induced from an elliptic fibration of type $({\rm III}^*, {\rm III})$ or
$({\rm III}^*, {\rm I}_2, {\rm I}_1)$.  The fiber of type ${\rm III}$ or type ${\rm I}_2$
is blown up and hence at least one of the remaining two circles is represented by a 
$(-1)$-root.  Now, consider the parabolic subdiagram of type
$\tilde{\rm A}_7+\tilde{A}_1$ which is induced from an elliptic fibration of type
$({\rm I}_8, {\rm I}_2, {\rm I}_1, {\rm I}_1)$.  Since the subdiagram of type $\tilde{A}_1$ 
consists of the two circles, the fiber of ${\rm I}_2$ is blown up.
Thus both of the two circles are represented by
$(-1)$-roots.  We have examples (Examples \ref{TypeI2}, \ref{TypeI1}) of such surfaces in any characteristic $p\ne 2$.

\smallskip

(2)\ {\bf The case of type II.}\quad There exist a parabolic subdiagram of type $\tilde{D}_8$ and one of type $\tilde{D}_5+\tilde{A}_3$ in Figure \ref{typeII}.
It follows from Table \ref{extremal} and Proposition \ref{Halphen} that 
$S$ is obtained from an elliptic fibration with singular fibers of type
$({\rm I}_4^*, {\rm I}_1, {\rm I}_1, {\rm I}_1)$ by blowing up the singular point of one or two fibers of type ${\rm I}_1$ and also from the one of type 
$({\rm I}_1^*, {\rm I}_4, {\rm I}_1)$ by blowing up the singular point(s) of 
one or two fibers of type ${\rm I}_1$, ${\rm I}_4$.  
This implies that the number $n$ of boundary components is 1.  
Thus there are no $(-1)$-roots and hence
all vertices of the dual graph are represented by $(-2)$-curves.  
We have an example (Examples \ref{TypeII}) of such a surface in any characteristic $p\ne 2$.

\smallskip

(3)\ {\bf The cases of type III and of type IV.}\quad   
The dual graph of type {\rm III} or of type {\rm IV} contains a parabolic subdiagram of type $\tilde{D}_4+\tilde{D}_4$ (\cite{Ko}, Figures 3.5, 4.4).  By Table \ref{extremal}, it is induced from an elliptic fibration with singular fibers of $({\rm I}_0^*, {\rm I}_0^*)$.  In this case no blow-ups occur and hence the cases of type III and IV do not occur.

\smallskip

(4)\ {\bf In case of type V.}\quad 
The dual graph of type V contains a parabolic subdiagram of type $\tilde{D}_6+\tilde{A}_1+\tilde{A}_1$ (Figure \ref{typeV}) which is induced from an elliptic fibration of type
$({\rm I}_2^*, {\rm I}_2, {\rm I}_2)$ (Table \ref{extremal}).  This implies $n=2$ or $n=4$.
On the other hand, the dual graph also contains a parabolic subdiagram of type 
$\tilde{E}_7+\tilde{A}_1$ which is induced from an elliptic fibration of type
$({\rm III}^*, {\rm I}_2, {\rm I}_1)$ or $({\rm III}^*, {\rm III})$.  
This implies $n=1, 2$ or $3$.  Thus we conclude $n=2$.
Note that the dual graph contains a parabolic subdiagram of type $\tilde{E}_6+\tilde{A}_2$ which is induced from an elliptic fibration of type $({\rm IV}^*, {\rm I}_3, {\rm I}_1)$ or 
$({\rm IV}^*, {\rm IV})$ in characteristic $p\ne 3$ and a quasi-elliptic fibration 
of type $({\rm IV}^*, {\rm IV})$ in $p=3$.  To get a Coble surface with $n=2$ the only possibility is the case where $p=3$: blowing up the singular points of two fibers of 
type II.
Thus this case occurs only in characteristic 3.  
It follows from Lemma \ref{D4diagram} and Lemma \ref{roots2}, (2) 
that all vertices of the subdiagram (A) in Figure \ref{typeV} are represented by $(-2)$-curves.  By using the fact $n=2$ and Lemma \ref{roots2}, 
we can see that all vertices of the subdiagram (B) are represented by $(-2)$-curves.
Finally consider a subdiagram of type $\tilde{A}_1$ in the subdiagram $(C)$ which 
is contained in a maximal parabolic subdiagram of type 
$\tilde{D}_6+\tilde{A}_1+\tilde{A}_1$.  
The subdiagram of type $\tilde{D}_6$ (resp. one of type $\tilde{A}_1$) belongs to the
subdiagram $(A)$ (resp. $(B)$).
These imply that all vertices in the subdiagram (C) are represented by $(-1)$-roots.  
We have an example of such a Coble surface (Example \ref{TypeV}).

\smallskip

(5) \ {\bf In case of type VI.}\quad 
First we assume $p\ne 3$.  The dual graph of type VI contains 
a parabolic subdiagram of type $\tilde{E}_6+\tilde{A}_2$ (Figure \ref{typeVI}) 
and one of type $\tilde{A}_4+\tilde{A}_4$.  The corresponding elliptic fibration 
is of type $({\rm IV}^*, {\rm I}_3, {\rm I}_1)$ or of type $({\rm I}_5, {\rm I}_5, {\rm I}_1, {\rm I}_1)$ ($({\rm I}_5, {\rm I}_5, {\rm II})$ in characteristic 5).
These imply $n=1$.
Then the $R$-invariant is $(E_6 \oplus A_4, \{ 0 \})$, and hence
$X$ is a supersingular $K3$ surface in characteristic 5 because of $|{\rm det}(A_4)|=5$. 
Since $n=1$, there are no $(-1)$-roots and hence all vertices of the dual graph are represented by $(-2)$-curves.  
We have an example of such a Coble surface (Example \ref{TypeVI5}).

Next assume $p=3$.  Note that there exists a quasi-elliptic fibration of type
$({\rm IV}^*, {\rm IV})$ or an elliptic fibration of type $({\rm IV}^*, {\rm I}_3)$ 
instead of $({\rm IV}^*, {\rm I}_3, {\rm I}_1)$.  
Also there exist parabolic subdiagrams of type 
$\tilde{A}_5+\tilde{A}_2+\tilde{A}_1$ and of type $\tilde{D}_5+\tilde{A}_3$ which correspond to an elliptic fibration of
type $({\rm I}_6, {\rm I}_3, {\rm III})$ and of type $({\rm I}_1^*, {\rm I}_4, {\rm I}_1)$,  respectively.  These imply that the case of type 
$({\rm IV}^*, {\rm I}_3)$ does not occur and $n=5$.  
It follows from Lemma \ref{D4diagram} that all vertices of the subdiagram (A) in Figure \ref{typeVI} are represented by $(-2)$-curves.  By Lemma \ref{roots2}, 
all vertices of the subdiagram (B) 
are represented by only $(-1)$-roots or by only $(-2)$-curves.  Since the fiber of type 
${\rm I}_4$ as above is blown up, the four vertices of $\tilde{A}_3$
in a subdiagram of type $\tilde{D}_5+\tilde{A}_3$ are represented by $(-1)$-roots.
Thus all vertices of the subdiagram (B) are represented by $(-1)$-roots. 
We have an example of such a Coble surface (Example \ref{TypeVI3}).

\smallskip

(6)\ {\bf In case of type VII.}\quad  
The dual graph contains parabolic subdiagrams of type $\tilde{A}_4+\tilde{A}_4$, 
$\tilde{A}_8$ and $\tilde{A}_5+\tilde{A}_2+\tilde{A}_1$ which correspond to
elliptic fibrations of type $({\rm I}_5, {\rm I}_5, {\rm I}_1, {\rm I}_1)$ ($({\rm I}_5, {\rm I}_5, {\rm II})$ in case $p=5$),
$({\rm I}_9,{\rm I}_1, {\rm I}_1, {\rm I}_1)$ ($({\rm I}_9,{\rm II})$ in case $p=3$), $({\rm I}_6,{\rm I}_3, {\rm I}_2, {\rm I}_1)$
($({\rm I}_6,{\rm I}_3, {\rm III})$ in case $p=3$) respectively.  These imply $n=1, 2$.
It follows from Lemma \ref{roots2} that all vertices of 
the subdiagram with 15 vertices and with single edges are realized by only $(-2)$-roots or 
by only $(-1)$-roots. Since $n=1, 2$, the second case does not occur.
Then we obtain $A_9$, as a component of the $R$-invariant $K$, consisting of $(-2)$-curves.
Hence $X$ is a supersingular $K3$ surface in characteristic 5 (Lemma \ref{overlattice}).  
On the other hand, the subdiagram of $\tilde{A}_4+\tilde{A}_4$ corresponds to an elliptic fibration of type $({\rm I}_5,{\rm I}_5, {\rm II})$ in characteristic 5.  Therefore we conclude $n=1$ and hence all vertices of the dual graph are represented by $(-2)$-curves.  
We have an example of such a Coble surface (Example \ref{TypeVII}).

\smallskip

(7)\ {\bf In case of type MI.}\quad 
We know all maximal parabolic subdiagrams of the dual graph (Lemma \ref{parabolicMI}) which
correspond to genus one fibrations with singular fibers of type $({\rm I}_6, {\rm I}_3, {\rm I}_2, {\rm I}_1)$ ($({\rm I}_6, {\rm I}_3, {\rm III})$ in case of $p=3$), $({\rm I}_5, {\rm I}_5, {\rm I}_1, {\rm I}_1)$ ($({\rm I}_5, {\rm I}_5, {\rm II})$ in case of $p=5$), $({\rm I}_4, {\rm I}_4, {\rm I}_2, {\rm I}_2)$, $({\rm I}_3, {\rm I}_3, {\rm I}_3, {\rm I}_3)$ ($({\rm IV}, {\rm IV}, {\rm IV}, {\rm IV})$ in case of $p=3$).
These imply that $n=6$ if $p\ne 3$ and $n=2$ if $p=3$.
If $n=6$, then we blow up all the singular points of the fiber of type ${\rm I}_6$ for any elliptic fibration of type
$({\rm I}_6, {\rm I}_3, {\rm I}_2, {\rm I}_1)$ which corresponds to the subdiagram of type
$\tilde{A}_5+\tilde{A}_2+\tilde{A}_1$.  
Recall that all vertices of the diagram of type $\tilde{A}_5$ are duads (resp. synthemes)
and all vertices of the diagram of type $\tilde{A}_2+\tilde{A}_1$ are synthemes or triads
(resp. duads or triads) (see the proof of Lemma \ref{parabolicMI}, (1)).  This implies that
all vertices of $\Gamma^t_{\rm MI}$ are represented by $(-2)$-curves.  Also we may assume that all vertices of $\Gamma^d_{\rm MI}$ are represented by $(-1)$-roots and those of
$\Gamma^s_{\rm MI}$ are represented by $(-2)$-curves.  Now consider a subdiagram of type
$\tilde{A}_3+\tilde{A}_3+\tilde{A}_1+\tilde{A}_1$ which is induced from an elliptic fibration of type $({\rm I}_4, {\rm I}_4, {\rm I}_2, {\rm I}_2)$.  One fiber of type 
${\rm I}_2$ is blown up and the other of type ${\rm I}_2$ is not.  On the other hand,
one diagram of type $\tilde{A}_1$ consists of a duad and a syntheme, and the other consists of triads (see the proof of Lemma \ref{parabolicMI}, (3)).  This is a contradiction.
Thus $n=2$ and $p=3$.  Again, by considering subdiagrams of type $\tilde{A}_5+\tilde{A}_2+\tilde{A}_1$ and of type $\tilde{A}_3+\tilde{A}_3+\tilde{A}_1+\tilde{A}_1$, we can see that all vertices of the subdiagrams of type $\Gamma^d_{\rm MI}$ and $\Gamma^s_{\rm MI}$ are represented by $(-2)$-curves and 
those of type $\Gamma^t_{\rm MI}$ are represented by $(-1)$-roots.  
We have an example of such a Coble surface (Example \ref{TypeMI}).

\smallskip

(8)\ {\bf In case of type MII.}\quad 
We know all maximal parabolic subdiagrams of the dual graph (Lemma \ref{parabolicMII}) which
correspond to genus one fibrations with singular fibers of type $({\rm I}_8, {\rm I}_2, {\rm I}_1, {\rm I}_1)$,
$({\rm I}_6, {\rm I}_3, {\rm I}_2, {\rm I}_1)$ ($({\rm I}_6, {\rm I}_3, {\rm III})$ in case of $p=3$), 
$({\rm I}_4, {\rm I}_4, {\rm I}_2, {\rm I}_2)$, $({\rm I}_3, {\rm I}_3, {\rm I}_3, {\rm I}_3)$ ($({\rm IV}, {\rm IV}, {\rm IV}, {\rm IV})$ in case of $p=3$).
These imply that $p=3$ and $n=2$ or $8$.
Assume $n=2$ and consider a quasi-elliptic fibration with reducible fibers of type 
$({\rm IV}, {\rm IV}, {\rm IV}, {\rm IV})$.  In this case we blow up the singular points of two irreducible fibers of type ${\rm II}$.  Since any vertex of the dual graph is contained in such a quasi-elliptic fibration, all 40 vertices of the dual graph are represented by $(-2)$-roots.
This is impossible because we blow up the singular points of a fiber of type ${\rm I}_2$
in an elliptic fibration of type $({\rm I}_4, {\rm I}_4, {\rm I}_2, {\rm I}_2)$.
Thus we have $p=3$ and $n=8$.  
Any vertex of $\Gamma_0$ is contained in a subdiagram of type $\tilde{A}_7$
which is induced from an elliptic fibration of type $({\rm I}_8, {\rm I}_2, {\rm I}_1, {\rm I}_1)$ (see Lemma \ref{parabolicMII} and its proof).  Since the fiber of type ${\rm I}_8$
is blown up, all vertices of $\Gamma_0$ are represented by $(-1)$-roots.
Again, by Lemma \ref{parabolicMII} and its proof, 
any vertex of $\Gamma_1\cup \Gamma_2$ corresponds to 
an irreducible component of a fiber not blown up.
Thus all vertices of the subdiagrams of type 
$\Gamma_1$ and $\Gamma_2$ are represented by $(-2)$-curves. 
We have an example of such a Coble surface (Example \ref{TypeMII}).

\section{Uniqueness}\label{sec7}

To finish the proof of Theorem \ref{mainth}, we will prove the uniqueness of each type of Coble surfaces
in Table \ref{main}.

In case $n=1$ (i.e. type I, II, VI, VII), the uniqueness follows from Martin \cite[Propositions 3.4, 4.3, 8.5, 9.4]{Martin}.

In case of type I with $n=2$, we can not apply Martin's argument because he used an elliptic fibration with singular fibers of type ${\rm III}^*, {\rm I}_2$, but in our case the corresponding fibration
is of type ${\rm III}^*, {\rm III}$.  However we can reverse the construction of the example given in
Example \ref{TypeI2} and obtain the same divisors on ${\bf P}^1\times {\bf P}^1$.  

In case of type V with $n=2$, we can also recover the divisors on ${\bf P}^1\times {\bf P}^1$ as in the previous case.

In case of type VI with $n=5$, consider an elliptic fibration defined by a dual graph of type $\tilde{A}_5+ \tilde{A}_2+ \tilde{A}_1$.  By contracting exceptional curves in the fibers of type $\tilde{A}_2$ and $\tilde{A}_1$, we obtain a Jacobian fibration with singular fibers of type ${\rm I}_6, {\rm I}_3,
{\rm III}$ (see Remark \ref{VIremark}).  Now the uniqueness of $S$ follows from the uniqueness of the Jacobian fibration of this type (Ito \cite{Ito}).  

In case of type MI, the covering $K3$ surface $X$ is the Fermat quartic surface and the covering involution is induced from a pair of skew lines on $X$.  Note that the defining equation of the Fermat quartic surface is a hermitian form over ${\bf F}_9$ and hence the projective unitary group ${\rm PGU}_4({\bf F}_9)$ acts on the set of pairs of two skew lines.  Recall that 
$|{\rm PGU}_4({\bf F}_9)|= 2^9\cdot 3^6\cdot 5\cdot 7$ (e.g. Atlas \cite[p.52]{ATLAS}).  
The number of pairs of skew lines is $112\cdot 81 /2=2^3\cdot 3^4 \cdot 7$ and the stabilizer subgroup  of a pair of skew lines is ${\rm Aut}(S) \times {\bf Z}/2{\bf Z}$ which has order $2^6\cdot 3^2\cdot 5$. Here ${\bf Z}/2{\bf Z}$ is generated by the covering involution of the double covering $X\to S$. 
Thus the action is transitive, and hence the uniqueness of $S$ follows.

In case of type MII, the covering $K3$ surface $X$ is a Kummer surface ${\rm Kum}(E\times E)$ associated with the unique supersingular elliptic curve $E$.  
Hence $S$ is unique, too.

Thus we have now finished the proof of Theorem \ref{mainth}.

\end{document}